%% file: main.tex
\documentclass{article}
\usepackage[utf8]{inputenx}
\usepackage{tikz-cd}
\usepackage{indentfirst}
\usepackage{microtype}
\usepackage{amsfonts}
\usepackage{mathtools}
\usepackage[english]{babel}
\usepackage{euscript}
\usepackage{setspace}
\usepackage{amsthm}
\usepackage[bottom]{footmisc}
\usepackage{bbm}
\usepackage{xpatch}
\usepackage{pgfplots}
\pgfplotsset{compat = 1.15}
\usetikzlibrary{arrows.meta}
\usepackage{amsmath,url}
\usepackage{wrapfig}
\usepackage{dsfont} 
\usepackage{caption}
\usepackage{float}
\usepackage{bm}
\usepackage [a4paper, top=3.5cm, bottom=3.5cm, left=3.3cm, right=3.3cm] {geometry}
\usepackage{comment}
\usepackage{todonotes}
\usepackage{amssymb}
\usepackage{overpic}
\usepackage{mathrsfs} 
\usepackage{hyperref}
\usepackage{makeidx}

\usepackage{resmes}

\usepackage{graphicx}
\usepackage[export]{adjustbox}
\graphicspath{ {./images/} }

\usepackage{diagbox}
\usepackage{makecell}

\usepackage{lettrine}
\theoremstyle{definition}
\newtheorem{definition}{Definition}[section]
\newtheorem{theorem}[definition]{Theorem}
\newtheorem{proposition}[definition]{Proposition}
\newtheorem{example}[definition]{Example} 
\newtheorem{rk}{Remark}[section]
\newtheorem{lemma}[definition]{Lemma}

\newtheorem{corollary}[definition]{Corollary}

\let\oldexample\example
\renewenvironment{example}
  {\oldexample\pushQED{\qed}\ignorespaces} 
  {\hfill$\blacklozenge$\par}                  

\usepackage{stackengine}

\newcommand{\jump}{\hfill\break}
\renewcommand{\inf}{\mathop{\mathrm{inf}\vphantom{\mathrm{sup}}}}

\newcommand{\R}{\mathbb{R}}
\newcommand{\N}{\mathbb{N}}
\newcommand{\C}{\mathbb{C}}
\newcommand{\Z}{\mathbb{Z}}
\newcommand{\T}{\mathbb{T}}

\newcommand{\e}{\textnormal{e}}
\renewcommand{\u}{\textnormal{u}}
\renewcommand{\i}{\textnormal{i}}
\renewcommand{\d}{\textnormal{d}}

\renewcommand{\L}{\operatorname{L}}
\renewcommand{\H}{\operatorname{H}}

\DeclareMathOperator{\Arg}{Arg}

\newcommand{\OA}{\mathcal{M}_{\mathrm{OA}}}
\newcommand*\samethanks[1][\value{footnote}]{\footnotemark[#1]}

\title{Unstable Manifolds for the Kuramoto Model:\\ Convergence to the Ott-Antonsen Manifold}
\author{ Christian Kuehn\thanks{Technical University of Munich, School of Computation, Information and Technology, Department of Mathematics, Boltzmannstraße 3, 85748 Garching, Germany. E-Mail: ckuehn@ma.tum.de, giacomo.landi@tum.de}, Giacomo Landi\samethanks}
\date{\today}

\begin{document}

\maketitle
\input{abstract}
\input{Intro}
\input{Spectral_Analysis}
\input{Unstable_Man}

\input{MFL_CL_connection}
\input{conclusion}
\bibliographystyle{abbrv}
\bibliography{zReference}
\input{Appendix}

\end{document}

%% file: abstract.tex
\begin{abstract}
In this paper, we study the finite-dimensional, homogeneous, all-to-all coupled Kuramoto model. We begin by performing a complete spectral analysis of all equilibria of the system. Motivated by this analysis, we then derive an explicit description of the unstable manifolds associated with the family of incoherent equilibria. Subsequently, we establish the convergence of this family of unstable manifolds to the Ott–Antonsen manifold $\OA$, with respect to the Hausdorff distance induced by the $p$-Wasserstein metric. We further carry out an analogous analysis for the corresponding counterpart of $\OA$ in the continuum limit. Moreover, we establish the uniform-in-time convergence of trajectories of the finite-dimensional Kuramoto model on these invariant manifolds towards their corresponding mean-field limit trajectories. Our results provide a direct geometric link between finite-dimensional particle systems and their mean-field, or continuum, limits.
\end{abstract}

%% file: Intro.tex
\begin{section}{Introduction}
Among interacting particle systems, the Kuramoto model has become one of the most important and influential models since its original formulation \cite{kuramoto2003chemical}. Its relevance stems from its ability to serve as a simple yet effective benchmark capable of accurately capturing real-world phenomena. In particular, the Kuramoto model has been widely used to describe systems of coupled oscillators; see \cite{kuramoto2005self}. From a mathematical perspective, the Kuramoto model consists of a system of $N$ ordinary differential equations given by
\begin{equation}\label{KM}
    \frac{\textnormal{d}\theta_i}{\textnormal{d}t}(t)=\dot{\theta}_i(t)=\omega_i+\frac{K}{N}\sum_{j=1}^Na_{ij}\sin{\left(\theta_j(t)-\theta_i(t)\right)},\qquad\forall i\in \{1,...,N\}=:[N]\tag{KM}
\end{equation}
where $\theta_i\in\mathbb{S}^1$ for all $i=1,...,N$ denotes the \textit{phase} of each oscillator, $\omega_i$ denotes the \textit{natural (or intrinsic) frequency} of each oscillator, $K\in\mathbb{R}^+$ is the \textit{coupling strength}, $N$ is the number of oscillators, and $\mathcal{A}=(a_{ij})_{i,j\in[N]}\in\R^{N\times N}$ is the adjacency matrix.\\

One of the most striking phenomena exhibited by \eqref{KM} is the long-time synchronization of the oscillators' phases. In the classical case of \textit{all-to-all} coupling, i.e., $a_{ij}=1$ for all $i,j=1,\ldots,N$, this phenomenon was studied in \cite{van1993lyapunov} using a Lyapunov function. Further results in this direction can be found in \cite{benedetto2014complete,dong2013synchronization,ha2010complete}, where a complete and rigorous analysis of synchronization is provided. For more general positive symmetric weights, this phenomenon was proved in \cite{jadbabaie2004stability}. When considering the Kuramoto model on oriented or signed graphs, the analysis becomes more difficult and new types of limiting behaviour may arise; see \cite{delabays2019kuramoto,burylko2014bifurcation,yoon2026stability,hong2011kuramoto}. Synchronization for the Kuramoto model perturbed by rough noise, a class of noises that includes fractional Brownian motion as a special case, has recently been established in \cite{neamtu2026synchronization}. Furthermore, a complete spectral analysis of all equilibria of the inhomogeneous Kuramoto model, i.e. when $\omega_i\neq\omega_j$, was carried out in \cite{mirollo2005spectrum}.\\

Another classical approach to problems related to the Kuramoto model, and more generally to particle systems, is to study the qualitative properties of the corresponding mean-field limit equation; see \cite{chiba2015proof}. The literature on mean-field limits is extensive; see \cite{braun1977vlasov,dobrushin1979vlasov,golse2016dynamics,jabin2018quantitative,neunzert2006approximation,serfaty2020mean,sznitman2006topics}.\\

In the case of the \textit{homogeneous} Kuramoto model, i.e. $\omega_i=\overline{\omega}$ for all $i$, and \textit{all-to-all} coupling, up to a change of coordinates, since the initial system is invariant under translations, \eqref{KM} can be rewritten as
\begin{equation}\label{KMH}
    \dot{\theta}_i(t)=\frac{K}{N}\sum_{j=1}^N\sin{(\theta_j(t)-\theta_i(t))}.\qquad\forall i=1,...,N\tag{KMH}
\end{equation}
The mean-field limit PDE associated with this particle system is given by
\begin{equation}\label{MFL}
\begin{cases}\tag{KMFL}
\displaystyle
    \partial_t \mu_t(\theta)+K\nabla_\theta\cdot\left(\mu_t(\theta)\int_{0}^{2\pi}\sin(\phi-\theta)\text{d}\mu_t(\phi)\right)=0,\\
    \mu_{t=0}=\mu_0.
\end{cases}
\end{equation}
with $\mu \in C([0,T]; \mathcal{P}(\mathbb{S}^1))$, where $\mu_t(\theta)$ describes the probability to find an oscillator at time $t$ at position $\theta$. The relationship between the particle system \eqref{KMH} and its mean-field limit equation \eqref{MFL} can be understood by introducing the empirical measure
$$\mu_t^N\coloneqq\frac{1}{N}\sum_{i=1}^N\delta_{\theta_i(t)},$$
which can be shown to satisfy \eqref{MFL}; see \cite[Th. 1.3.1]{golse2016dynamics}. The convergence of the microscopic dynamics $(\theta_i)_{i=1}^N$ governed by \eqref{KMH} towards the mean-field solution $\mu$ of \eqref{MFL} can be established using a stability argument; see \cite{golse2016dynamics,dobrushin1979vlasov}. This argument yields an estimate of the form
\begin{equation}\label{MFL_rate_conv}
    W_1(\mu_t,\mu_t^N)\leq C(t)W_1(\mu_0,\mu_0^N),
\end{equation}
where $W_1$ is the $1$-Wasserstein distance \cite{dobrushin1979vlasov}. We refer also to~\cite{loeper2006uniqueness}, where Wasserstein distances with exponents different from $1$ are used in the same context. For a detailed introduction to Wasserstein distances, we refer the reader to \cite{villani2008optimal}.\\

Another possible way of approximating particle systems such as \eqref{KMH} is via the continuum limit. In this case, one obtains pointwise convergence of the solutions of \eqref{KMH} towards $x(t,\xi)\in C([0,T];L^\infty(I,[0,2\pi]))$, the solution of the following integro-differential equation:
\begin{equation}\label{CL}
\begin{cases}\tag{KCL}
\displaystyle
    \partial_t x(t,\xi)=K\int_0^1 \sin\big(x(t,\xi)-x(t,z)\big)\textnormal{d}z,\\
    x(0,\xi)=x_0(\xi).
\end{cases}
\end{equation}
This limit is of comparable importance, particularly in the context of network dynamics and graph limits; see \cite{medvedev2014nonlinear,medvedev2018continuum,kuehn2019power,gkogkas2021continuum}.\\

A well-established approach to studying the qualitative behaviour of a dynamical system consists in analyzing the linear stability of its steady states (or equilibrium points). In particular, the dynamics in a neighbourhood of an equilibrium is determined by the associated local invariant manifolds, namely the stable, unstable, and center manifolds. For finite-dimensional systems, these topics are treated extensively in classical textbooks; we refer, for instance, to \cite{perko2013differential}. For partial differential equations, the corresponding theory is also well-developed; see, for example, \cite{sell2002dynamics, henry2006geometric, bates1998existence}.
It is straightforward to observe that all trajectories of the system lie in a hyperplane, since the mean phase
\begin{equation}
    \phantomsection
    \Theta\coloneqq\frac{1}{N}\sum_{i=1}^N\theta_i(t)
    \label{mean_phase}
\end{equation}
is a first integral of the system. In \cite{maistrenko2005chaotic,maistrenko2005desynchronization,popovych2005phase}, it was shown that if the natural frequencies are symmetrically distributed, namely $\omega_i=-\omega_{N-i+1}$, then the Kuramoto model \eqref{KM} admits an invariant manifold given by $\mathcal{M}=\{\theta_i=-\theta_{N-i+1}\},$ which is an $[N/2]$-dimensional torus. Under additional assumptions on the natural frequencies and for small coupling strength $K$, the stability of $\mathcal{M}$ was analyzed in \cite{chiba2009stability}. Yet, the analysis of all invariant manifolds of Kuramoto models is far from being complete. One of the main contributions of the present work is the explicit parametrization of the \textit{global} unstable manifold $W^\u(\theta^N_q)$, where $\theta^N_q$ is an element of the family of incoherent equilibria, given by
$$\theta^N_q\coloneqq\left(\frac{\pi(1-N)}{N}+q,\ldots,\frac{\pi(2i-1-N)}{N}+q,\ldots,\frac{\pi(N-1)}{N}+q\right)\in\T^N\qquad \forall q\in\T.$$

On the other hand, when considering the mean-field limit equation associated with a particle system, one may study the invariant manifolds of the corresponding limit PDE. In this context, a particularly relevant manifold is the so-called Ott--Antonsen (OA) manifold $\OA$. The OA manifold was introduced in the seminal work \cite{ott2008low} and has since been extensively studied; see, for example, \cite{tyulkina2018dynamics,ott2009long,engelbrecht2020ott,martens2009exact}. It provides a remarkable finite-dimensional reduction of the mean-field dynamics for large ensembles of coupled phase oscillators with a suitable interaction kernel. In particular, for the PDE \eqref{MFL}, the manifold $\OA$ can be written as
\begin{equation}\label{OA_Man}
    \OA=\left\{f_{\alpha,\beta}(\theta)=\frac{1}{2\pi}\frac{1-\beta^2}{1-2\beta\cos(\alpha+\theta)+\beta^2}\;\bigg\vert\;\beta\in[0,1),\;\alpha\in[-\pi,\pi]\right\}
\end{equation}
and the dynamics restricted to it reduces to a finite-dimensional system of ODEs given by
\begin{equation}\label{red_dyn_OA}
    \begin{cases}
        \dot{\beta}=\frac{K}{2}\beta(1-\beta^2),\\
        \dot{\alpha}=0.
    \end{cases}
\end{equation}
These expressions can be derived by following the same strategy as in \cite{ott2008low} for non-identical oscillators. We recall that the elements of $\OA$ are probability densities on $[0,2\pi]$ and that
$$W^\u\left(\frac{1}{2\pi}\right)=\OA,$$
see Appendix \ref{App_C} for further details.\\

In the recent work \cite{kuehn2025mean}, we introduced a transformation that allows one to construct solutions of \eqref{CL} starting from solutions of \eqref{MFL}. This transformation relies on the pseudo-inverse of the cumulative distribution function associated with a measure, namely
$$F^{-1}_\mu(\xi)=\inf\left\{x\in\R\;\bigg\vert\; F_\mu(x)\geq\xi\right\}.$$
Applying this transformation to the incoherent equilibrium $\frac{1}{2\pi}$ of \eqref{MFL} yields the family of incoherent equilibria
$$y_q(\xi)\coloneqq2\pi\xi+q$$
for \eqref{CL}, with $q\in\T$ and $\xi\in[0,1]$. In the same work, by means of this transformation, we also identified the analogue of the manifold $\OA$ for \eqref{CL}, namely
\begin{equation}\label{CL_OA_Man}
     W^\u(y_q(\xi))=\left\{g_{\alpha,\beta}(\xi)=F^{-1}_{\alpha,\beta}\left(\xi+C(\alpha,\beta)+\frac{q}{2\pi}\right)\bigg\vert\;\beta\in[0,1),\;\alpha\in[-\pi,\pi]\right\}
\end{equation}
with
$$C(\alpha,\beta)=-\frac{1}{\pi}\arctan\left(\frac{\beta\sin(\alpha)}{1-\beta\cos(\alpha)}\right)$$
and
\begin{equation}\label{inv_CDF_1}
     F^{-1}_{\alpha,\beta}(\xi)=\begin{cases}
        G_{\alpha,\beta}(\xi)\qquad\qquad\; \text{if}\;\xi\in[0,c(\alpha,\beta)]\\
        G_{\alpha,\beta}(\xi)+2\pi\qquad \text{if}\;\xi\in[c(\alpha,\beta),1],
    \end{cases}
\end{equation}
where
\begin{equation}\label{inv_CDF_2}
    G_{\alpha,\beta}(\xi)\coloneqq2\arctan\left(\frac{1-\beta}{1+\beta}\tan\left(\pi\xi+\arctan\left(\frac{1+\beta}{1-\beta}\tan\left(\frac{\alpha}{2}\right)\right)\right)\right)-\alpha
\end{equation}
and
$$c(\alpha,\beta)\coloneqq\frac{1}{2}-\frac{1}{\pi}\arctan\left(\frac{1+\beta}{1-\beta}\tan\left(\frac{\alpha}{2}\right)\right).$$
In particular, the reduced dynamics on $W^\u(y_q(\xi))$ coincides with \eqref{red_dyn_OA}.\\

The main contribution of this paper is to establish the convergence of $W^\u(\theta^N_q)$ towards $\OA$ and $W^\u(y_q(\xi))$, with respect to the Hausdorff distance induced, respectively, by the $p$-Wasserstein distance and the $L^\infty$ distance. To the best of the authors’ knowledge, this is the first work proving the convergence of a sequence of invariant manifolds associated with a ``\textit{natural}'' particle system towards an invariant manifold of its mean-field limit. By the term ``\textit{natural}'', we refer to particle systems that are canonically studied in the literature and from which limit equations are derived as $N\to\infty$.\\

Furthermore, the convergence of the invariant manifolds is complemented by a uniform-in-time mean-field limit for trajectories lying on $W^u(\theta_q^N)$. In particular, the corresponding approximation error remains uniformly controlled for all times, in contrast with the classical mean-field theory, where the error typically grows exponentially in time.\\

The paper is organized as follows. In Section \ref{Spectr_Sec}, we perform a complete spectral analysis of all equilibria of \eqref{KMH}, thereby filling a gap in the literature left by \cite{mirollo2005spectrum}. In Section \ref{main_sec}, we provide an explicit parametrization of the global unstable manifold $W^\u(\theta^N_q)$. In Section \ref{conn_sec}, we establish the convergence of $W^\u(\theta^N_q)$ towards $\OA$ and $W^\u(y_q(\xi))$, together with an explicit linear rate of convergence. Finally, in Section \ref{concl}, we summarize our results and discuss several directions for future research. Auxiliary formulas and additional technical details are collected in Appendix \ref{APP}.
\end{section}

%% file: Spectral_Analysis.tex
\begin{section}{Spectral Analysis}\label{Spectr_Sec}\noindent

The goal of this section is to provide a complete spectral analysis of all equilibria of \eqref{KMH}. This consists in determining all eigenvalues, together with their multiplicities and associated eigenspaces, for every equilibrium point $\theta^\star$. This complements the analysis of \cite{mirollo2005spectrum}, since in that work the authors considered only the inhomogeneous Kuramoto model and explicitly used the presence of two distinct natural frequencies $\omega_i\neq\omega_j$ in their proof.

In what follows, we use the following notation. The homogeneous Kuramoto vector field is denoted by $V(\theta)=(V_1(\theta),\ldots,V_N(\theta))$, where
$$V_i(\theta)=\frac{K}{N}\sum_{j=1}^N\sin(\theta_j-\theta_i)$$
for all $i=1,\ldots,N$ and $\theta\in\T^N$. Moreover, we denote by $\e_1,\ldots,\e_N$ the elements of the canonical basis of $\R^N$, and by $\underline{1}=\{1\}^N$ the $N$-dimensional vector whose entries are all equal to $1$. For $\Tilde{N},\Tilde{M}\in\{1,\ldots,N\}$, we write $\mathbbm{1}_{\Tilde{N},\Tilde{M}}=\{1\}^{\Tilde{N}\times \Tilde{M}}$ for the matrix whose entries are all equal to $1$, and $\mathbb{I}_{\Tilde{N}}\in\R^{\Tilde{N}\times \Tilde{N}}$ for the identity matrix. We suppress the index whenever it is clear from the context the dimension of the matrices.

We know from \cite[Proposition 5]{nguyen2023equilibria} that $\theta^\star\in\T^N$ is an equilibrium point of \eqref{KMH} if and only if $\theta^\star$ satisfies one of the following conditions:
\begin{description}
    \item[\textbf{A})\label{A}] the components $\theta^\star_i$ differ by integer multiples of $\pi$,
    \item[\textbf{B})\label{B}] $\sum_{j=1}^Ne^{\i\theta^\star_j}=0$, where $\i=\sqrt{-1}$.
\end{description}
With this in mind, we introduce the following classification of equilibrium points.

\begin{definition}
Given $\theta^\star\in\T^N$, we say that it is:
    \begin{description}
        \item[i)\label{i}] a \textit{coherent equilibrium} if there exists $q\in\T$ such that $\theta^\star=(q,\ldots,q)$;
        \item[ii)\label{ii}] a \textit{$J$-partially coherent equilibrium} if there exist $q\in\T$ and $J=1,\ldots,\lfloor N/2\rfloor$ such that
        $$\theta^\star=\left(\underbrace{q+\pi,\ldots,q+\pi}_{J},q,\ldots,q\right);$$
        \item[iii)\label{iii}] an \textit{incoherent equilibrium} if there exists $q\in\T$ such that
        $$\theta^\star=\left(\frac{\pi}{N}+q,\ldots,\frac{\pi(2i-1)}{N}+q,\ldots,\pi\frac{2N-1}{N}+q\right);$$
        \item[iv)\label{iv}] a \textit{general incoherent equilibrium} if $\theta^\star$ satisfies \nameref{B}.
    \end{description}
\end{definition}

We observe that this is simply another way of classifying the equilibria listed in \nameref{A} and \nameref{B}. Moreover, it is sufficient to study the stability of the $J$-partially coherent equilibria in order to analyze all equilibria of type \nameref{A}. Indeed, up to permutations, which are allowed since all particles are indistinguishable, and up to periodicity, one obtains all equilibrium points for which the components $\theta^\star_i$ differ by integer multiples of $\pi$. Furthermore, it is easy to see that the family of general incoherent equilibria contains all incoherent equilibria and some of the $J$-partially coherent equilibria.

\begin{rk}
    As already mentioned, it is known that the coherent equilibrium is locally stable; see \cite{arenas2008synchronization}. Moreover, it was shown in \cite{taylor2012there} that for \eqref{KMH} there are no other stable equilibrium points. In the same paper, the author proved that the same statement holds for homogeneous Kuramoto models with a positive adjacency matrix $\mathcal{A}$ whose node degrees are at least $\mu(N-1)$, with $\mu=0.9395$. In general, there exists a critical connectivity value $\mu_c$ such that if $\mu>\mu_c$, then the coherent equilibrium attracts almost every point of the system. It is conjectured that $\mu_c=0.75$. In recent years, the following results have been obtained in this direction: in \cite{ling2019landscape}, it was proved that $\mu_c\leq0.7929$; this bound was first refined in \cite{lu2020synchronization} to $\mu\leq0.7889$, and then in \cite{kassabov2021sufficiently} to $\mu_c\leq0.75$; finally, in \cite{yoneda2021lower}, the lower bound $\mu_c>0.6838$ was established.
\end{rk}

It is straightforward to observe that the Jacobian matrix associated with $V$ is given by
\begin{equation}\label{gen_jacobian}
    DV(\theta)_{i j}= \frac{K}{N}\begin{cases}\displaystyle-\sum_{\substack{k=1 \\ k \neq i}}^N \cos \left(\theta_k-\theta_i\right) & \text { if } i=j \\ \cos \left(\theta_j-\theta_i\right) & \text { if } i \neq j.\end{cases}
\end{equation}
We are now ready to state the following result.

\begin{proposition}\label{prop_spectral}
    For the homogeneous Kuramoto model with all-to-all coupling \eqref{KMH}, the following statements hold:
    \begin{description}
        \item[i)] A coherent equilibrium $\theta^\star$ has
        \begin{itemize}
            \item[$\bullet$] $\lambda_1=0$ as an eigenvalue, with associated eigenspace given by $\operatorname{span}\{\underline{1}\}$;
            \item[$\bullet$] $\lambda_2=-K$ as an eigenvalue, with associated eigenspace given by $\operatorname{span}\{\e_2-\e_1,\cdots,\e_j-\e_1,\cdots,\e_N-\e_1\}$.
        \end{itemize}

        \item[ii)] A $1$-partially coherent equilibrium $\theta^\star$ has
        \begin{itemize}
            \item[$\bullet$] $\lambda_1=0$ as an eigenvalue, with associated eigenspace given by $\operatorname{span}\{\underline{1}\}$;
            \item[$\bullet$] $\lambda_2=K$ as an eigenvalue, with associated eigenspace given by $\operatorname{span}\{\underline{1}-N\e_1\}$;
            \item[$\bullet$] $\lambda_3=K\frac{2-N}{N}$ as an eigenvalue, with associated eigenspace given by $\operatorname{span}\{\e_3-\e_2,\cdots,\e_j-\e_2,\cdots,\e_N-\e_2\}$.
        \end{itemize}

        \item[iii)] A $J$-partially coherent equilibrium $\theta^\star$, with $J=2,\ldots,\lfloor N/2\rfloor$ when $N$ is odd or $J=2,\ldots,\lfloor N/2\rfloor-1$ when $N$ is even, has
        \begin{itemize}
            \item[$\bullet$] $\lambda_1=0$ as an eigenvalue, with associated eigenspace given by $\operatorname{span}\{\underline{1}\}$;
            \item[$\bullet$] $\lambda_2=K$ as an eigenvalue, with associated eigenspace given by $\operatorname{span}\{\underline{1}-N/J(\e_1+\cdots+\e_J)\}$;
            \item[$\bullet$] $\lambda_3=K\frac{N-2J}{N}$ as an eigenvalue, with associated eigenspace given by $\operatorname{span}\{\e_2-\e_1,\cdots,\e_j-\e_1,\cdots,\e_J-\e_1\}$;
            \item[$\bullet$] $\lambda_4=K\frac{2J-N}{N}$ as an eigenvalue, with associated eigenspace given by $\operatorname{span}\{\e_{J+2}-\e_{J+1},\cdots,\e_{j}-\e_{J+1},\cdots,\e_{N}-\e_{J+1}\}$.
        \end{itemize}

        \item[iv)] A $\lfloor N/2\rfloor$-partially coherent equilibrium $\theta^\star$ with $N$ even has
        \begin{itemize}
            \item[$\bullet$] $\lambda_1=0$ as an eigenvalue, with associated eigenspace given by $\operatorname{span}\{\e_2-\e_1,\cdots,\e_{\lfloor N/2\rfloor}-\e_1,\e_{\lfloor N/2\rfloor+1}+\e_1,\cdots,\e_N+\e_1\}$;
            \item[$\bullet$] $\lambda_2=K$ as an eigenvalue, with associated eigenspace given by $\operatorname{span}\{-\e_1-\cdots-\e_{\lfloor N/2\rfloor}+\e_{\lfloor N/2\rfloor+1}+\cdots+\e_{N}\}$.
        \end{itemize}

        \item[v)] A general incoherent equilibrium $\theta^\star$ such that there exist two indices $i_0$ and $i_1$ satisfying $\theta^\star_{i_0}-\theta^\star_{i_1}\notin\pi\Z$ has
        \begin{itemize}
            \item[$\bullet$] $\lambda_1=0$ as an eigenvalue, with associated eigenspace given by an $(N-2)$-dimensional subspace;
            \item[$\bullet$] 
            $$\lambda_{2,3}=\frac{K}{2}\left(1\pm\frac{1}{N}\left|\sum_{j=1}^N\exp^{\i2\theta^\star_j}\right|\right),$$
            and both eigenvalues have a one-dimensional associated eigenspace.
        \end{itemize}
    \end{description}
\end{proposition}

\begin{proof}
    The proof of points \textbf{i)}-\textbf{iv)} is essentially a direct computation. Indeed, for \textbf{i)} one has
    $$DV(\theta^\star)=\frac{K}{N}(\mathbbm{1}-N\mathbb{I}),$$
    while for \textbf{ii)}-\textbf{iv)}, with $N\geq2$ and for all $J=1,\cdots,\left\lfloor\frac{N}{2}\right\rfloor$, one has
$$DV(\theta^\star)=\frac{K}{N}\left(
\begin{array}{c|c}
\begin{array}{c}
\mathbbm{1}_J+(N-2J)\mathbb{I}_J\\
\end{array}
&
\begin{array}{c}
-\mathbbm{1}\\
\end{array}
\\
\hline
\begin{array}{c}
-\mathbbm{1}\\
\end{array}
&
\begin{array}{c}
\mathbbm{1}_{N-J}+(2J-N)\mathbb{I}_{N-J}\\
\end{array}
\end{array}
\right).$$
It is therefore straightforward to verify that these matrices have the stated eigenvalues and associated eigenspaces. The last case is more interesting. Indeed, for \textbf{v)}, we cannot explicitly compute the Jacobian in a simple form. Nevertheless, we can compute its action on arbitrary vectors in $\R^N$. In what follows, we denote
    $$c_i\coloneqq\cos(\theta^\star_i),\qquad s_i\coloneqq\sin(\theta^\star_i),$$
    and by $\underline{c}$ and $\underline{s}$ the vectors whose entries are, respectively, the $c_i$s and the $s_i$s. We now start by showing that $\underline{c}$ and $\underline{s}$ are linearly independent. Indeed, if
    $$a\underline{c}+b\underline{s}=\underline{0}$$
    with $a,b\in\R$, then necessarily $a=b=0$. This follows from the fact that, in particular, one would have
    $$\begin{cases}
        ac_{i_0}+bs_{i_0}=0,\\
        ac_{i_1}+bs_{i_1}=0,
    \end{cases}$$
    which admits no solution other than $a=b=0$, since $\theta^\star_{i_0}-\theta^\star_{i_1}\notin\pi\Z$. Taking a general vector $\underline{v}\in\R^N$, we obtain
    \begin{align*}
        \left(DV(\theta^\star)\underline{v}\right)_i
        &=\frac{K}{N}\sum_{j=1}^N\cos(\theta^\star_j-\theta^\star_i)(v_j-v_i)\\
        &=\frac{K}{N}\sum_{j=1}^N(c_jc_i+s_js_i)(v_j-v_i)\\
        &=\frac{K}{N}(c_i\underline{c}\cdot\underline{v}+s_i\underline{s}\cdot\underline{v}),
        \end{align*}
where we have used $\underline{c}\cdot\underline{1}=\underline{s}\cdot\underline{1}=0$, which follows from \nameref{B}. From the latter expression and the linear independence of $\underline{c}$ and $\underline{s}$, we already see that a vector $\underline{v}$ belongs to $\operatorname{ker}(DV(\theta^\star))$ if and only if $\underline{c}\cdot\underline{v}=\underline{s}\cdot\underline{v}=0$. We can then complete $\{\underline{c},\underline{s}\}$ to a basis of $\R^N$ by selecting vectors $\e_3,\ldots,\e_N$ that are mutually orthogonal and orthogonal to both $\underline{c}$ and $\underline{s}$. In these new coordinates, the Jacobian can be written as
$$DV(\theta^\star)=\frac{K}{N}\left(\begin{array}{ccccc}
\|\underline{c}\|^2 & \underline{c}\cdot\underline{s} & 0 & \cdots & 0 \\
\underline{c}\cdot\underline{s} & \|\underline{s}\|^2 & 0 & \cdots & 0 \\
0 & 0 & 0 & \cdots & 0 \\
\vdots & \vdots & \vdots & \ddots & \vdots \\
0 & 0 & 0 & \cdots & 0
\end{array}\right),$$
from which the statement of the lemma follows.
\end{proof}

\begin{rk}\label{rk_eigenspace}
    We observe that $\lambda_2,\lambda_3\in(0,K)$, where the endpoints are excluded because the conditions
$$\frac{1}{N}\left|\sum_{j=1}^N\exp^{\i2\theta^\star_j}\right|=1$$
and \nameref{B} are incompatible. Moreover, we remark that when $\theta^\star$ is the incoherent solution, one obtains
$$\lambda_2=\lambda_3=\frac{K}{2}.$$
For this particular incoherent equilibrium, an alternative proof based on circulant matrices is provided in Appendix \ref{App_B}. Moreover, for every incoherent equilibrium, the eigenspace associated with $\lambda_2=K/2$ is spanned by
$$E^\u(\theta^N_q)=\operatorname{span}\left\{\left(\begin{array}{c}
0 \\
\sin(\frac{2\pi}{N})\\
\vdots \\
\sin(\frac{2\pi}{N}j)\\
\vdots \\
\sin(\frac{2\pi}{N}(N-1))
\end{array}\right),\;\left(\begin{array}{c}
1 \\
\cos(\frac{2\pi}{N})\\
\vdots \\
\cos(\frac{2\pi}{N}j)\\
\vdots \\
\cos(\frac{2\pi}{N}(N-1))
\end{array}\right)\right\}.$$
\end{rk}


\end{section}

%% file: Unstable_Man.tex
\begin{section}{Explicit parametrization of $W^\u(\theta^N_q)$}\label{main_sec}
We are now interested in characterizing the unstable manifold $W^\u(\theta^N_q)$ for every $q\in\T$. We first derive a discrete analogue of \cite[Lemma 2.5]{kuehn2025mean}, which describes how the vector field in \eqref{KMH} acts on our ansatz, namely a discretized version of \eqref{CL_OA_Man}. To this end, for every $(\alpha,\beta,A,\xi)\in[-\pi,\pi]\times[0,1)\times\R\times[0,1]$ we define the function
\begin{equation}\label{continuous_ansatz}
    \phi(\alpha,\beta,A,\xi)\coloneqq2\arctan\left(\frac{1-\beta}{1+\beta}\tan\left(\pi\xi+A\right)\right)-\alpha+2\pi k
\end{equation}
where $k$ is chosen such that $\pi\xi+A\in[\pi(k-1/2),\pi(k+1/2)]$. We are now ready to state the following technical lemma.
\begin{lemma}\label{discrete_lemma}
For every $\alpha\in[-\pi,\pi]$, $\beta\in[0,1)$, and $A\in\R$, define, for every $i=1,\ldots,N$,
\begin{equation}\label{fin_N_ansatz}
    \phi_i(\alpha,\beta,A)\coloneqq\phi(\alpha,\beta,A,\xi_i)
\end{equation}
with
$$\xi_i\coloneqq\frac{2i-1}{2N}.$$
Then, for every $i=1,\ldots,N$, one has
\begin{equation}\label{discrete_formula}
    \frac{1}{N}\sum_{j=1}^N\sin\left(\phi_j-\phi_i\right)=-\beta\sin(\phi_i+\alpha)B_N(A,\beta)-\beta\cos(\phi_i+\alpha)D_N(A,\beta),
\end{equation}
where
$$B_N(A,\beta)\coloneqq\frac{1+\beta^{2N-2}+\cos(2AN)(-\beta)^{N-2}(1+\beta^2)}{1+\beta^{2N}+2(-\beta)^N\cos(2AN)}$$
and
$$D_N(A,\beta)\coloneqq\frac{\sin(2AN)(-\beta)^{N-2}(\beta^2-1)}{1+\beta^{2N}+2(-\beta)^N\cos(2AN)}.$$
\end{lemma}

\begin{proof}
    We begin by observing that
    $$\frac{1}{N}\sum_{j=1}^N\sin\left(\phi_j-\phi_i\right)=\textnormal{Im}\left(e^{-\i \phi_i} S_N\right),$$
    where
    $$S_N\coloneqq\frac{1}{N}\sum_{j=1}^Ne^{\i \phi_j}.$$
    By using the identity
    $$e^{\i 2\arctan(y)}=\frac{1+\i y}{1-\i y} \qquad\forall y\in\R$$
    and its equivalent formulation
    $$\i \tan\left(y\right)=\frac{e^{2\i y}-1}{e^{2\i y}+1}\qquad\forall y\in\R,$$
    we obtain
    \begin{equation}\label{exp_expression}
        e^{\i \phi_i}=e^{-\i \alpha}\frac{1+\i\frac{1-\beta}{1+\beta}\tan\left(\pi\xi_i+A\right)}{1-\i\frac{1-\beta}{1+\beta}\tan\left(\pi\xi_i+A\right)}=e^{-\i \alpha}\frac{Ce^{2\i \pi\xi_i}+\beta}{1+C\beta e^{2\i\pi\xi_i}},  
    \end{equation}
    where $C\coloneqq e^{2\i A}$. It remains to compute
    $$S_{N,1}\coloneqq\sum_{j=1}^N \frac{1}{1+C\beta \eta_j}$$
    and
    $$S_{N,2}\coloneqq\sum_{j=1}^N \frac{\eta_j}{1+C\beta \eta_j},$$
    where $\eta_j\coloneqq e^{2\pi\i\xi_j}$, so that $\eta_j^N=-1$ for every $j=1,\ldots,N$. In fact, it is sufficient to compute $S_{N,1}$, since
    \begin{equation}\label{S_N12}
        S_{N,2}=\frac{N}{C\beta}-\frac{1}{C\beta}S_{N,1}.
    \end{equation}
    In particular, the points $C\eta_j^N$ are zeros of the complex polynomial
    $$P(z)\coloneqq z^N+C^N.$$
    Hence, by the fundamental theorem of algebra,
    $$\frac{Nz^{N-1}}{z^N+C^N}=\frac{P'(z)}{P(z)}=\sum_{j=1}^N\frac{1}{z-C\eta_j}\qquad\forall z\in\C.$$
    Choosing $z=-1/\beta$, the previous identity yields
    $$-\beta S_{N,1}=\frac{N(-1)^{1-N}\beta^{1-N}}{(-1)^{-N}\beta^{-N}+C^N}=N\frac{(-1)^{N-1}\beta}{(-1)^N+(\beta C)^N}=-N\frac{\beta}{1+(-\beta C)^N}.$$
    Consequently, using \eqref{S_N12}, we deduce that
    $$S_{N,2}=\frac{N}{C\beta}-\frac{N}{C\beta} \frac{1}{1+(-\beta C)^N}=\frac{N}{C\beta}\frac{(-\beta C)^N}{1+(-\beta C)^N}.$$
    By combining the previous identities, we obtain
    \begin{align*}
        S_N=\frac{e^{-\i \alpha}}{N}\Big(CS_{N,2}+\beta S_{N,1}\Big)
        &=e^{-\i \alpha}\left(\frac{(-\beta C)^N}{\beta(1+(-\beta C)^N)}+\frac{\beta}{1+(-\beta C)^N}\right)\\
        &=e^{-\i \alpha}\beta\left(\frac{1+C^N(-\beta)^{N-2}}{1+(-\beta C)^N}\right).
    \end{align*}
    By the definition of $C$, the last factor can be rewritten as
    \begin{align*}
        \frac{1+C^N(-\beta)^{N-2}}{1+(-\beta C)^N}
        &=\frac{\left(1+(-\beta)^Ne^{-2\i AN}\right)\left(1+(-\beta)^{N-2}e^{2\i AN}\right)}{\left(1+(-\beta)^Ne^{-2\i AN}\right)\left(1+(-\beta)^Ne^{2\i AN}\right)}\\
        &=\frac{1+\beta^{2N-2}+(-\beta)^Ne^{-2\i AN}+(-\beta)^{N-2}e^{2\i AN}}{1+\beta^{2N}+2(-\beta)^N\cos(2AN)}.
    \end{align*}
    Since the last quantity has been expressed as a linear combination of complex exponentials, its imaginary part can be computed explicitly. This gives
    \begin{align*}
        \frac{1}{N}\sum_{j=1}^N\sin(\phi_j-\phi_i)&=-\frac{\beta}{1+\beta^{2N}+2(-\beta)^N\cos(2AN)}\Big[\sin(\phi_i+\alpha)\left(1+\beta^{2N-2}\right)+\\
        &\hspace{1.9cm}+\sin(\phi_i+\alpha+2AN)(-\beta)^N+\sin(\phi_i+\alpha-2AN)(-\beta)^{N-2}\Big]=\\
        &\hspace{-2.5cm}=-\frac{\beta}{1+\beta^{2N}+2(-\beta)^N\cos(2AN)}\bigg[\sin(\phi_i+\alpha)\Big(1+\beta^{2N-2}+\cos(2AN)(-\beta)^{N-2}(1+\beta^2)\Big)+\\
        &\hspace{5.0cm}+\cos(\phi_i+\alpha)\sin(2AN)(-\beta)^{N-2}(\beta^2-1)\bigg]
    \end{align*}
    which concludes the proof.
\end{proof}

We remark that the addition of the factor $2\pi k$ plays no role in the previous proof. It is included only to ensure the continuity of the functions $\phi_i$. Before proceeding with the main theorem, we would like to prove a short technical lemma about the $\phi_i$, which will be used in the proof of the main theorem.
\begin{lemma}
    For every $N\in\N$, $\alpha\in[-\pi,\pi]$, $\beta\in[0,1)$ and $A\in\R$ consider the sequence given by $\big(\phi_i(\alpha,\beta,A)\big)_{i=1}^N$, with the functions $\phi_i(\cdot)$ defined as in \eqref{fin_N_ansatz}. Then for every $i=1,\ldots,N-1$ we have that
    \begin{equation}\label{phi_prop_1}
        \phi_i(\alpha,\beta,A)\leq\phi_{i+1}(\alpha,\beta,A)
    \end{equation}
    and
    \begin{equation}\label{phi_prop_2}
        0\leq\phi_N(\alpha,\beta,A)-\phi_1(\alpha,\beta,A)\leq2\pi
    \end{equation}
\end{lemma}
\begin{proof}
 The first inequality is immediate, indeed for fixed $\alpha$, $\beta$, and $A$, the map $\xi \mapsto \phi(\alpha,\beta,A,\xi)$ is monotone increasing. Hence, since $\xi_i\leq\xi_{i+1}$ for every $i=1,\ldots,N-1$, we obtain $\phi_i(\alpha,\beta,A)\leq\phi_{i+1}(\alpha,\beta,A).$ To prove the second inequality, note that
$$0=\phi_1(\alpha,\beta,A)-\phi_1(\alpha,\beta,A)\leq\phi_N(\alpha,\beta,A)-\phi_1(\alpha,\beta,A)\leq\phi(\alpha,\beta,A,1)-\phi(\alpha,\beta,A,0)=2\pi,$$ 
where we used \eqref{phi_prop_1}. This concludes the proof.
\end{proof}

Having established the previous lemma, we can now describe the local unstable manifold $W^u(\theta^N_q)$, as stated in the following theorem.

\begin{theorem}\label{main_theorem}
    Every element of the family of incoherent equilibria $\theta^N_q$ of \eqref{KMH}, with $q\in\mathbb{T}$ and $N\geq3$, admits an explicit expression for its two-dimensional unstable manifold, namely
    $$W^\u(\theta^N_q)=\left\{\theta\in\T^N\text{ s.t. } \theta_i=\phi_i(\alpha,\beta,A);\text{ }\frac{1}{N}\sum_{i=1}^N\theta_i=q\text{ }\bigg\vert\;\beta\in[0,1),\;\alpha\in[-\pi,\pi],A\in\left[-\frac{\pi}{2},\frac{\pi}{2}\right]\right\}.$$
    Moreover, $(\beta,\alpha,A)$ satisfy the following system of ODEs:
    \begin{equation}\label{3d-system}
    \begin{cases}
        \dot{\beta}=\frac{K}{2}\beta(1-\beta^2)B_N(A,\beta),\\
        \dot{\alpha}=\frac{K}{2}(1+\beta^2)D_N(A,\beta),\\
        \dot{A}=\frac{K}{4}(1-\beta^2)D_N(A,\beta).
    \end{cases}
    \end{equation}
\end{theorem}

\begin{proof}
For clarity, throughout this proof we denote the set appearing on the right-hand side of the statement by $\mathcal{M}$. The proof proceeds in four steps. First, we show that, as a set of $\R^N$, $\mathcal{M}$ is contained is some hypercube, namely
$$\mathcal{M}\subseteq\theta^N_q+(-\pi,\pi)^N.$$
Therefore this will allow us to well define the mean phase on $\R^N$ and, in a second step, to project $\mathcal{M}$ on the $N$-dimensional torus $\T^n$ without creating any topological pathology. Then, we prove that $\mathcal{M}\subseteq W^\u(\theta^N_q)$. We continue by proving that, upon restricting $\mathcal{M}$ to a sufficiently small neighbourhood $\mathcal{U}$, one has
$$W^\u_{loc}(\theta^N_q)\cap\mathcal{U}=\mathcal{M}\cap\mathcal{U}.$$ Finally, using the invariance of $\mathcal{M}$ under the dynamics, we obtain the reverse inclusion $W^\u(\theta^N_q)\subseteq\mathcal{M}$, thereby completing the proof.
\jump\jump\noindent
\fbox{\textbf{Step 1:} $\mathcal{M}\subseteq\theta^N_q+(-\pi,\pi)^N$} We start by noticing that the quantity
$$\frac{1}{N}\sum_{i=1}^N\phi_i(\alpha,\beta,A)$$
is well defined on $\R^N$ and therefore we can use it in order to reduce the number of variables from three to two. Indeed, if we define $\varphi_i(\beta,A)\coloneqq\phi_i(\alpha,\beta,A)+\alpha$ and fix $q\in\R$, one has
\begin{equation}\label{3to2}
    \alpha=-q+\frac{1}{N}\sum_{i=1}^N\varphi_i(\beta,A).
\end{equation}
This relation uniquely determines $\alpha$ as a function of $\beta$ and $A$, which in the following we will denote by $\alpha_N(\beta,A)$ or simply by $\alpha(\beta,A)$ or $\alpha$ when the value of $N$ is clear from the context, depending on what we want to emphasize. In particular, we see that $i$-th component of $\theta\in\mathcal{M}$ is given by
$$\theta_i=q+\varphi_i(\beta,A)-\frac{1}{N}\sum_{j=1}^N\varphi_j(\beta,A)=q+\frac{1}{N}\sum_{j=1}^N\big(\varphi_i(\beta,A)-\varphi_j(\beta,A)\big).$$
It is clear from the definition of $\varphi_i$ that the properties \eqref{phi_prop_1} and \eqref{phi_prop_2} hold true for them as well. Therefore, by using these properties we obtain that
$$\theta_i\leq q+\frac{1}{N}\sum_{j=1}^{i-1}\big(\varphi_i(\beta,A)-\varphi_j(\beta,A)\big)\leq q+\frac{2\pi}{N}(i-1)$$
and, analogously, that
$$\theta_i\geq q+\frac{1}{N}\sum_{j=i+1}^{N}\big(\varphi_i(\beta,A)-\varphi_j(\beta,A)\big)\geq q-\frac{2\pi}{N}(N-i).$$
From this, and the definition of $\theta^N_q$, we obtain that
$$\theta_i-\left(\theta^N_q\right)_i\leq\frac{2\pi}{N}(i-1)+q-\frac{\pi}{N}(2i-1-N)-q=\pi\left(1-\frac{1}{N}\right)<\pi$$
and, analogously, that
$$\theta_i-\left(\theta^N_q\right)_i\geq\frac{2\pi}{N}(N-i)+q-\frac{\pi}{N}(2i-1-N)-q=-\pi\left(1-\frac{1}{N}\right)>-\pi,$$
which concludes the first step of the proof. From now on we continue our analysis on $\T^N$ since we are working on a proper subset of it and therefore it is the same as working on $\R^N$.\jump\jump\noindent
 \fbox{\textbf{Step 2:} $\mathcal{M}\subseteq W^\u(\theta^N_q)$} We now prove that the whole set $\mathcal{M}$ is contained in the global unstable manifold $W^u(\theta^N_q)$. To this end, we begin by deriving the ODE system \eqref{3d-system}. Firstly, we apply the classical chain rule to the functions $\theta_i$, obtaining
$$\dot{\theta}_i=-\dot{\alpha}-2\frac{\sin(\theta_i+\alpha)}{1-\beta^2}\dot{\beta}+2\frac{1-2\beta\cos(\theta_i+\alpha)+\beta^2}{1-\beta^2}\dot{A}.$$
Detailed computations leading to this identity can be found in Lemma \ref{chain_rule}, in Appendix \ref{App_A}. Comparing this equation with \eqref{discrete_formula}, we see that the triplet $(\beta,\alpha,A)$ must satisfy
    \begin{equation}
    \begin{cases}
    \displaystyle
        -\frac{2}{1-\beta^2}\dot{\beta}=-K\beta B_{N}(A,\beta),\\
        \displaystyle
        -\frac{4\beta}{1-\beta^2}\dot{A}=-K\beta D_{N}(A,\beta),\\
        \displaystyle
        -\dot{\alpha}+2\frac{1+\beta^2}{1-\beta^2}\dot{A}=0,
    \end{cases}
    \end{equation}
which is equivalent to \eqref{3d-system}. In particular, this shows that the set $\mathcal{M}$ is invariant under the dynamics \eqref{KMH}. As already mentioned during Step 1, by using the first integral \eqref{mean_phase}, we can uniquely express $\alpha$ as a function of $\beta$ and $A$, denoted by $\alpha(\beta,A)$. On the $N$-dimensional torus $\T^N$, we have that \eqref{3to2} reads as
\begin{equation}\label{mean_phase_new}
    -\alpha+\frac{2}{N}\sum_{i=1}^N\arctan\left(\frac{1-\beta}{1+\beta}\tan\left(\pi\xi_i+A\right)\right)=q.
\end{equation}
We now observe that $\theta^N_q\in\mathcal{M}$. Indeed, for every value of $A$, one has
\begin{align*}
        \phi_i(\alpha(0,A),0,A)
        &=2\pi\xi_i+2A-\alpha(0,A)\\
        &=2\pi\xi_i+2A-\frac{2}{N}\sum_{j=1}^N(\pi\xi_j+A)+q\\
        &=2\pi\xi_i-\pi+q=(\theta^N_q)_i.
\end{align*}
 We can now analyze the equation for $\dot{\beta}$ and deduce that $\mathcal{M}\subseteq W^\u(\theta^N_q)$. Our goal is to prove that
    $$\underset{t\rightarrow-\infty}{\lim}\beta(t)=0.$$
so that consequently, we obtain
    $$\underset{t\rightarrow-\infty}{\lim}\phi(\alpha,\beta,A)(t)=\theta^N_q,$$
and hence, by definition, $\phi(\alpha,\beta,A)\in W^u(\theta^N_q)$. We first prove that
    $$B_N(A,\beta)\geq\frac{N-2}{N}\qquad \forall A\in\left[-\frac{\pi}{2},\frac{\pi}{2}\right],\forall\beta\in[0,1).$$
Indeed, recalling that
    $$B_N(A,\beta)=\frac{1+\beta^{2N-2}+\beta^{N-2}(1+\beta^2) r}{1+\beta^{2N}+2\beta^N r},$$
where $r=(-1)^N\cos(2NA)$, we see that, for every fixed $\beta\in[0,1)$, differentiating with respect to $r$ shows that this function is strictly increasing in $r$. Hence, its minimum is attained at $r=-1$. Therefore,
    $$B_N(A,\beta)\geq\frac{1+\beta^{2N-2}-\beta^{N-2}(1+\beta^2)}{1+\beta^{2N}-2\beta^N}=\frac{1-\beta^{N-2}}{1-\beta^N}\geq\frac{N-2}{N}$$
for every $A$ and $\beta$. In particular, for every $N\geq3$ and every initial datum $(\beta,\alpha,A)\in(0,1)\times[-\pi,\pi]\times\left[-\frac{\pi}{2},\frac{\pi}{2}\right]$, one has $\dot{\beta}>0$. Hence, there exists $\ell\in[0,1)$ such that
    $$\underset{t\rightarrow-\infty}{\lim}\beta(t)=\ell.$$
By the invariance of $\alpha$-limit sets, one must have $\ell=0$. This proves that $\mathcal{M}\subset W^u(\theta^N_q)$.\jump\jump\noindent
\fbox{\textbf{Step 3:} $W^\u_{loc}(\theta^N_q)\cap\mathcal{U}=\mathcal{M}\cap\mathcal{U}$}
We next show that, locally, $\mathcal{M}$ is a two-dimensional analytic manifold with tangent space at $\theta^N_q$ given by $E^\u\left(\theta^N_q\right)$, as in Remark \ref{rk_eigenspace}. To this end, it is convenient to rewrite the elements of $\mathcal{M}$ in Cartesian coordinates. Using \eqref{mean_phase_new} and the trigonometric identity \eqref{trig_2}, we obtain
\begin{align*}
        \phi_i(\alpha,\beta,A)
        &=2\pi\xi_i+2A-2\arctan\left(\frac{\beta\sin(2\pi\xi_i+2A)}{1+\beta\cos(2\pi\xi_i+2A)}\right)-\pi-2A\\
        &\hspace{5cm}
        +\frac{2}{N}\sum_{j=1}^N\arctan\left(\frac{\beta\sin(2\pi\xi_j+2A)}{1+\beta\cos(2\pi\xi_j+2A)}\right)+q\\
        &=\left(\theta^N_q\right)_i
        -2\arctan\left(\frac{x s_i+yc_i}{1+xc_i-ys_i}\right)
        +\frac{2}{N}\sum_{j=1}^N\arctan\left(\frac{x s_j+yc_j}{1+xc_j-ys_j}\right),
    \end{align*}
where we have set $c_i=\cos(2\pi\xi_i)$, $s_i=\sin(2\pi\xi_i)$, and
    $$\begin{cases}
        x=\beta\cos(2A),\\
        y=\beta\sin(2A).
    \end{cases}$$
Thus, we obtain a new parametrization of our set in terms of $(x,y)\in\R^2$ with $x^2+y^2<1$. Denoting by $\hat{\phi}_i(x,y)$ the function $\phi_i(\alpha,\beta,A)$ written in these Cartesian coordinates, a direct computation gives
    \begin{equation}\label{tangency_comp}
        \partial_x\hat{\phi}_i(0,0)=-2s_i,\qquad\partial_y\hat{\phi}_i(0,0)=-2c_i.
    \end{equation}
    Since the vectors $\underline{c}$ and $\underline{s}$ are linearly independent, as already shown in the proof of Proposition \ref{prop_spectral}, the Jacobian matrix of $\hat{\phi}\coloneqq(\hat{\phi}_1,\ldots,\hat{\phi}_N)$ at the origin has rank two, namely
    $$\operatorname{rk}\left(D\hat{\phi}(0,0)\right)=2.$$
    Hence, there exists a $2\times2$ minor with non-zero determinant. By continuity, there exists an open subset $\Tilde{\mathcal{U}}\subseteq B_1((0,0))\subset\R^2$ containing the origin such that, for every $(x,y)\in\Tilde{\mathcal{U}}$,
    $$\operatorname{rk}\left(D\hat{\phi}(x,y)\right)=2.$$
    Since $\hat{\phi}$ is analytic, the Rank Theorem \cite{narasimhan1965lectures} implies, after possibly restricting $\Tilde{\mathcal{U}}$, that
    $$\hat{\phi}(\Tilde{\mathcal{U}})\subset\T^N$$
    is a two-dimensional analytic submanifold. Moreover, since $(0,0)\in\Tilde{\mathcal{U}}$, then we have that $\theta^N_q\in\hat\phi(\Tilde{\mathcal{U}})$. From Step 2 we know that $\hat\phi(\Tilde{\mathcal{U}})\subseteq\mathcal{M}\subseteq W^\u(\theta^N_q)$ so that, by the Rank Theorem, $\hat\phi(\Tilde{\mathcal{U}})$ is an open set with respect to the subset topology induced by $W^\u(\theta^N_q)$. Therefore, there must exist an open subset $\mathcal{U}$ such that
    $$\mathcal{M}\cap\mathcal{U}=W^\u_{loc}(\theta^N_q)\cap\mathcal{U}$$
    which concludes this step.\jump\jump\noindent
\fbox{\textbf{Step 4:} $W^\u(\theta^N_q)\subseteq\mathcal{M}$} It remains to prove the reverse inclusion. As anticipated, this follows from the local identification of $\mathcal{M}$ with $W^u_{loc}(\theta^N_q)$ and from the invariance of $\mathcal{M}$. Indeed, let $\theta\in W^u(\theta^N_q)$. Then
    $$\underset{t\rightarrow-\infty}{\lim}\theta(t)=\theta^N_q,$$
so there exists $t_0<0$ such that $\theta(t_0)\in\mathcal{U}$. Since $W^u(\theta^N_q)$ is invariant, we have
    $$\theta(t_0)\in\mathcal{U}\cap W^u(\theta^N_q)=\mathcal{U}\cap W^u_{loc}(\theta^N_q)=\mathcal{M}\cap\mathcal{U}.$$
Thus, $\theta(t_0)\in\mathcal{M}$, and evolving this point forward in time yields $\theta\in\mathcal{M}$. This completes the proof.
\end{proof}

It is also worth noting that the manifolds $W^\u(\theta^N_q)$ are closely related to the Watanabe-Strogatz ansatz \cite{watanabe1993integrability,watanabe1994constants}. In those works, the authors proved that a system of $N$ phase oscillators with global cosine coupling is completely integrable and that the $N$-dimensional phase space is foliated by two-dimensional manifolds. The transformation relating their original variables to those of the reduced system has the same form as the functions $\phi_i$ in \eqref{fin_N_ansatz}. In their setting, the number of variables is reduced from three to two because two of their variables form a Hamiltonian system with one degree of freedom. In our case, this reduction is instead made possible by the existence of another first integral, namely the mean phase $\Theta$ in \eqref{mean_phase}.
\end{section}

%% file: MFL_CL_connection.tex
\begin{section}{Connection with CL and MFL}\label{conn_sec}
We are now interested in showing how this family of manifolds is related to the Ott--Antonsen manifold and to its counterpart discovered in \cite{kuehn2025mean}. We begin with a lemma which provides the infinite-dimensional analogue of the constraint imposed by the first integral of the system.

\begin{lemma}\label{inf_constrain}
    For every $\alpha\in[-\pi,\pi]$, $\beta\in[0,1)$, and $A\in[-\pi/2,\pi/2]$, one has
    $$\lim_{N\rightarrow+\infty}\frac{1}{N}\sum_{i=1}^N\phi_i(\alpha,\beta,A)=-\alpha+2A+\pi.$$
\end{lemma}

\begin{proof}
    Fix $\alpha$, $\beta$, and $A$, and denote the function $\phi(\alpha,\beta,A,\xi)$, defined in \eqref{continuous_ansatz}, simply by $\phi(\xi)$.
     From the continuity of $\phi(\xi)$ and from the definition of the Riemann integral, it follows that
    $$\lim_{N\rightarrow+\infty}\frac{1}{N}\sum_{i=1}^N\phi_i(\alpha,\beta,A)=\int_0^1\phi(z)\d z=-\alpha+\frac{1}{\pi}\int_A^{A+\pi}\Tilde{\phi}(z)\d z,$$
    where
    $$\Tilde{\phi}(z)\coloneqq\phi\left(\frac{z-A}{\pi}\right)+\alpha.$$
    We readily observe that
    $$\Tilde{\phi}(z+\pi)=\Tilde{\phi}(z)+2\pi,\qquad\Tilde{\phi}(-z)=-\Tilde{\phi}(z)$$
    for all $z\in\R$. Therefore,
    $$\hat{\phi}(z)\coloneqq\Tilde{\phi}(z)-2z$$
    is $\pi$-periodic and odd, hence
    $$\int_{A}^{A+\pi}\hat{\phi}(z)=0,$$
    which implies
    $$\lim_{N\rightarrow+\infty}\frac{1}{N}\sum_{i=1}^N\phi_i(\alpha,\beta,A)=-\alpha+\frac{1}{\pi}\int_A^{A+\pi}2z\d z=-\alpha+2A+\pi.\qedhere$$
\end{proof}
We also observe that the dynamics on $W^u(\theta^N_q)$ simplifies in the limit $N\rightarrow+\infty$. Indeed, since $\beta\in[0,1)$, the system \eqref{3d-system} converges to
\begin{equation}\label{limit_dyn}
    \begin{cases}
    \dot{\beta}=\frac{K}{2}\beta(1-\beta^2),\\
    \dot{\alpha}=0,\\
    \dot{A}=0,
\end{cases}
\end{equation}
which is precisely the reduced dynamics on $\mathcal{M}_{OA}$. From now on we are going to denote the solutions of \eqref{3d-system} with $(\alpha_N(t),\beta_N(t),A_N(t)$, in order to differentiate them from the one of the latter 3-dimensional system.\\

It remains to show how the geometric structure of $W^u(\theta^N_q)$ is approximated by $\mathcal{M}_{OA}$ and by its counterpart for the continuum limit as $N\rightarrow+\infty$. Therefore we now introduce two operators that lift finite-dimensional vectors into infinite-dimensional spaces, such as spaces of functions or measures. The first one is the \textbf{\textit{empirical-lift}} operator
\begin{align*}
    \mu^N:&\;\R^N\longrightarrow\mathcal{P}_1\\
             &\;x\longrightarrow\mu^N(x)\coloneqq\frac{1}{N}\sum_{i=1}^N\delta_{x_i}.
\end{align*}
while the second operator is the \textbf{\textit{step-function lift}} operator
\begin{align*}
    \L_N:&\;\R^N\longrightarrow\mathcal L^\infty([0,1])\\
             &\;x\longrightarrow\L_N(x)(\cdot),
\end{align*}
where
$$\L_N(x)(y)=x_i\quad\text{if }y\in\mathcal{I}_i$$
and
$$\mathcal{I}_i\coloneqq\left[\frac{i-1}{N},\frac{i}{N}\right)\text{ for }i=1,\ldots,N-1\qquad\text{ and }\qquad \mathcal{I}_N=\left[\frac{N-1}{N},1\right].$$
Of course, both operators can be defined analogously on the torus $\T^N$. In what follows, we denote by $W_p(\cdot,\cdot)$ the $p$-Wasserstein distance for $p>0$; see \cite{villani2008optimal} for further details. Moreover, we denote by $d_{\H}^{p}(\cdot,\cdot)$ the Hausdorff distance in the metric space $(\mathcal{P}_1,W_p)$, and by $d_{\H}^{\infty}(\cdot,\cdot)$ the Hausdorff distance in the metric space $(\mathcal L^\infty,\|\cdot\|_\infty)$. In the following theorem, we denote the restrictions of the previous manifolds to the range $\beta\in[0,\Tilde{\beta}]$ by
$$W^\u(\theta^N_q)_{\Tilde{\beta}},\qquad \mathcal{M}_{\mathrm{OA},{\Tilde{\beta}}},\qquad W^\u(y_q)_{\Tilde{\beta}}.$$

\begin{theorem}\label{conv_manifolds}
For every $\Tilde{\beta}\in[0,1)$, $N\geq3$, $q\in\T$, and every $p\geq1$, there exist positive constants $C_p=C_p(\tilde{\beta})>0$, $C=C(\tilde{\beta})>0$ such that
\begin{equation}\label{est_MFL}
d_{\H}^p\left(\mu^N\left(W^\u(\theta^N_q)_{\Tilde{\beta}}\right),\mathcal{M}_{\mathrm{OA},{\Tilde{\beta}}}\right)\leq\frac{C_p}{N},
\end{equation}
and
\begin{equation}\label{est_CL}
    d_{\H}^{\infty}\left(\L_N\left(W^u(\theta^N_q)_{\Tilde{\beta}}\right),W^\u(y_{q-\pi})_{\Tilde{\beta}}\right)\leq\frac{C}{N}.
\end{equation}
\end{theorem}
\begin{proof}
Since we have explicitly identified the three manifolds that we aim to compare, the proof reduces to obtaining explicit uniform estimates for their elements. We begin by fixing $\Tilde{\beta}\in[0,1)$ and estimating the Hausdorff distance between $W^\u(\theta^N_q)_{\Tilde{\beta}}$ and $\mathcal{M}_{\mathrm{OA},{\Tilde{\beta}}}$. Let $f_{\alpha,\beta}\in\OA$, it follows from \eqref{inv_CDF_1} and \eqref{inv_CDF_2} that the inverse of its cumulative distribution function $F^{-1}_{\alpha,\beta}$ is differentiable. Therefore, by the proof of \cite[Theorem 5.15]{xu2019best}, we have
$$W_p\left(\frac{1}{N}\sum_{i=1}^N\delta_{F^{-1}_{\alpha,\beta}\left(\frac{2i-1}{N}\right)},f_{\alpha,\beta}\d x\right)\leq\frac{1}{2N(p+1)^{1/p}}\max_{\xi\in[0,1]}\partial_\xi F^{-1}_{\alpha,\beta}(\xi),$$
for every $N\in\N$, $\alpha\in[-\pi,\pi]$, and $\beta\in[0,1)$. By definition of $F^{-1}_{\alpha,\beta}$, we know that
$$\partial_\xi F^{-1}_{\alpha,\beta}(\xi)=\frac{1}{f_{\alpha,\beta}(F^{-1}_{\alpha,\beta}(\xi))},$$
which implies that
$$\sup_{(\alpha,\beta)\in[-\pi,\pi]\times[0,\Tilde{\beta}]}W_p\left(\frac{1}{N}\sum_{i=1}^N\delta_{F^{-1}_{\alpha,\beta}\left(\frac{2i-1}{N}\right)},f_{\alpha,\beta}\d x\right)\leq \frac{M_p}{N}$$
for some $M_p>0$. Moreover, from the definition of $W_p$, we see that for every $c\in\R$ and every $N\in\N$,
$$W_p\left(\frac{1}{N}\sum_{i=1}^N\delta_{F^{-1}_{\alpha,\beta}\left(\frac{2i-1}{N}\right)},\frac{1}{N}\sum_{i=1}^N\delta_{F^{-1}_{\alpha,\beta}\left(\frac{2i-1}{N}+c\right)}\right)\leq\frac{L_{\alpha,\beta}}{N}\leq\frac{L}{N},$$
where $L_{\alpha,\beta}$ denotes the Lipschitz constant of $F^{-1}_{\alpha,\beta}$ and
$$L\coloneqq\sup_{\alpha,\beta}L_{\alpha,\beta}.$$
We observe that the previous inequality remains valid if $c$ is replaced by a generic sequence $c_N$. In particular, since the previous estimates are uniform on $[-\pi,\pi]\times[0,\Tilde{\beta}]$, and since a typical element $\theta\in W^\u(\theta^N_q)$ has components $\theta_i=\phi_i(\alpha,\beta,A)$ which, after expressing $A$ as $A_N(\alpha,\beta)$ \eqref{A_express}, are of the form $F^{-1}_{\alpha,\beta}(\frac{2i-1}{N}+c)$ for some $c$, we obtain
$$d_{\H}^p\left(\mu^N\left(W^\u(\theta^N_q)_{\Tilde{\beta}}\right),\mathcal{M}_{\mathrm{OA},{\Tilde{\beta}}}\right)\leq\frac{L+M_p}{N}.$$
This proves the first claim of the theorem, with $C_p\coloneqq L+M_p$.

We now prove the second statement. We again use the differentiability of $F^{-1}_{\alpha,\beta}$ for every $(\alpha,\beta)\in[-\pi,\pi]\times[0,\Tilde{\beta}]$. By the mean value theorem, we obtain
\begin{align*}
\left\|\L_N\left(\underline{F}^{-1}_{\alpha,\beta}\right)-F^{-1}_{\alpha,\beta}\right\|_\infty=\sup_{i=1,\ldots,N}\sup_{\xi\in\mathcal{I}_i}\left|F^{-1}_{\alpha,\beta}(\xi_i)-F^{-1}_{\alpha,\beta}(\xi)\right|\leq\frac{\left\|\partial_\xi F^{-1}_{\alpha,\beta}\right\|_\infty}{N},
\end{align*}
where the $i$-th component of the vector $\underline{F}^{-1}_{\alpha,\beta}$ is given by ${F}^{-1}_{\alpha,\beta}(\xi_i)$. By applying \eqref{trig_2}, we see that a typical element $g_{\alpha,\beta}$ of $W^\u(y_{q-\pi})$, given in \eqref{CL_OA_Man}, can be re-written as
$$g_{\alpha,\beta}(\xi)=2\arctan\left(\frac{1-\beta}{1+\beta}\tan\left(\pi\xi+\frac{q}{2}+\frac{\alpha}{2}-\frac{\pi}{2}\right)\right)-\alpha.$$
Thus, denoting by $\alpha_N(\beta,A)$ the expression for $\alpha$ obtained from \eqref{mean_phase_new}, we have
\begin{align*}
        \bigg\|g_{-q+2A+\pi,\beta}-\L_N\left(\underline{\phi}(\alpha_N(\beta,A),\beta,A)\right)\bigg\|_\infty\leq\big|q-2A+\pi-\alpha_N(\beta,A)\big|+\\
        +\left\|2\arctan\left(\frac{1-\beta}{1+\beta}\tan\left(\pi\xi+A\right)\right)-\L_N\left(\underline{\phi}(\alpha_N(\beta,A),\beta,A)\right)+\alpha_N(\beta,A)\right\|_{\infty}\leq\\
        \leq \frac{\left\|\partial_\xi F^{-1}_{\alpha_{N},\beta}\right\|_\infty}{4N}+\frac{\left\|\partial_\xi F^{-1}_{\alpha_{N},\beta}\right\|_\infty}{N}
    \end{align*}
Here, the first term is bounded by the rate of convergence of the corresponding Riemann sums towards their Riemann integral. Indeed, for every Lipschitz function $h:[0,1]\rightarrow\R$, one has
\begin{align*}
        \left|\frac{1}{N}\sum_{i=1}^N h\left(\xi_i\right)-\int_0^1h(z)\d z\right|=&\left|\sum_{i=1}^N\int_{\frac{i-1}{N}}^{\frac{i}{N}} (h\left(\xi_i\right)-h(z))\d z\right|\leq\operatorname{Lip}(h)\sum_{i=1}^N\int_{\frac{i-1}{N}}^{\frac{i}{N}} |\xi_i-z|\d z=\\
        =&\operatorname{Lip}(h)\sum_{i=1}^N\frac{1}{4N^2}=\frac{\operatorname{Lip}(h)}{4N}.
    \end{align*}
Therefore,
$$\sup_{(A,\beta)\in[-\pi/2,\pi/2]\times[0,\Tilde{\beta}]}\bigg\|g_{-q+2A+\pi,\beta}-\L_N\left(\underline{\phi}(\alpha_N(\beta,A),\beta,A)\right)\bigg\|_\infty\leq \frac{C}{N}.$$
As before, this implies the desired estimate for the Hausdorff distance:
$$d_{\H}^{\infty}\left(\L_N\left(W^u(\theta^N_q)_{\Tilde{\beta}}\right),W^\u(y_{q-\pi})_{\Tilde{\beta}}\right)\leq\frac{C}{N}.$$
\end{proof}
\begin{rk} The fact that \eqref{est_MFL} holds with respect to every Wasserstein distance $W_p$ is not surprising, since on bounded domains all Wasserstein distances are equivalent, that is, $W_p$ and $W_q$ are equivalent for every $p,q\geq 1$. Moreover, by the monotonicity of $L^p$ norms on probability spaces, \eqref{est_CL} implies convergence with respect to the Hausdorff distance induced by any $L^p$ norm, for every $p\geq 1$. Hence, the latter result does not essentially depend on the choice of ambient topology.\\

The step-function lift operator $\operatorname{L}_N$ is introduced to remain consistent with the continuum-limit procedure commonly adopted in the literature; see \cite{medvedev2014nonlinear}. For our purposes, however, a linear interpolation would lead to the same result. Higher-order interpolations could also be employed, owing to the smoothness of $F^{-1}_{\alpha,\beta}(\xi)$.
\end{rk}
The convergence rates established in Theorem \ref{conv_manifolds} suggest that trajectories evolving on the manifolds identified above may exhibit a stronger mean-field approximation than that provided by the classical theory \cite{golse2016dynamics}. Indeed, in the classical setting, the convergence factor $C(t)$ in \eqref{MFL_rate_conv} grows proportionally to $\e^{tL}$, where $L$ denotes the Lipschitz constant of the interaction kernel; for \eqref{KMH}, one has $L=1$. We will show that, when the initial datum lies on the unstable manifolds described above, the corresponding particle trajectories converge to the mean-field dynamics uniformly in time.\\

Our strategy relies on two complementary ingredients. For values of $\beta$ bounded away from $1$, the geometric convergence of the invariant manifolds provides the desired approximation. In contrast, when $\beta$ approaches $1$, we exploit the fact that both the finite-dimensional and infinite-dimensional reduced dynamics converge exponentially towards the synchronized state. Combining these two observations will eventually allow us to establish a uniform-in-time mean-field limit. We therefore begin by deriving quantitative estimates for the convergence towards the synchronized state.
\begin{theorem}\label{conv_rate_synchro}
\begin{enumerate} 
\item[i)] For every $(\alpha(0),\beta(0))\in[-\pi,\pi]\times(0,1)$, the trajectory $f_{\alpha(t),\beta(t)}\in\OA$ converges exponentially towards $\delta_{\alpha(0)}$ with respect to the $p$-Wasserstein distance. More precisely, for every $p\geq1$ there exists a constant $\Tilde{C}_p>0$ such that, for every initial datum $(\alpha(0),\beta(0))\in[-\pi,\pi]\times(0,1)$ and every $t\in\R^+$, the following estimates hold:
\begin{equation} W_p\big(f_{\alpha(t),\beta(t)}\d\theta,\delta_{\alpha(0)}\big)^p \leq \Tilde{C}_p\frac{1-\beta(0)^2}{\beta(0)^2} e^{-Kt}, \qquad \forall p>1; \end{equation} 
\begin{equation} W_1\big(f_{\alpha(t),\beta(t)}\d\theta,\delta_{\alpha(0)}\big) \leq \Tilde{C}_1 \frac{1-\beta(0)^2}{\beta(0)^2} e^{-Kt} \left( 1+\left|\ln\frac{1-\beta(0)^2}{\beta(0)^2}\right|+Kt \right). 
\end{equation} 
\item[ii)] For every $q\in\T^1$, $N\geq3$ and $\theta(0)\in W^\u(\theta^N_q)$, one has \[ \lim_{t\rightarrow+\infty}|\theta(t)-q\underline{1}|=0. \] Moreover, for every $p\geq1$ there exists a constant $\Tilde{D}_p>0$ such that, for every $t\in\R$, the following estimates hold:
\begin{equation} W_p\left(\mu^N(\theta(t)),\delta_{q}\right)^p \leq \Tilde{D}_p \left( \frac{1}{N} + \frac{1-\beta(0)^2}{\beta(0)^2} e^{-\frac{K}{3}t} \right), \qquad \forall p>1; \end{equation}
\begin{equation} W_1\left(\mu^N(\theta(t)),\delta_{q}\right) \leq \Tilde{D}_1 \left( \frac{1}{N} + \frac{1-\beta(0)^2}{\beta(0)^2} e^{-\frac{K}{3}t} \left( 1+\left|\ln \frac{1-\beta(0)^2}{\beta(0)^2}\right| +\frac{K}{3}t \right) \right).
\end{equation}
\end{enumerate}
\end{theorem}
    \begin{proof} \textbf{i)} By the translational invariance of the Wasserstein distance $W_p$ induced by the Euclidean metric, we may assume without loss of generality that $\alpha=0$. Since the product measure is an admissible coupling between $f_{0,\beta}\d\theta$ and $\delta_0$, we obtain 
    $$W_p(f_{0,\beta}\d\theta,\delta_0)^p\leq\int_{-\pi}^\pi|\theta|^pf_{0,\beta}(\theta)\d\theta =2\int_0^\pi\theta^p\frac{1-\beta^2}{2\pi(1-2\beta\cos\theta+\beta^2)}\,\d\theta.$$
    For $\beta\in[0,1/2]$, it follows immediately that
    \begin{equation}\label{low_beta} W_p(f_{0,\beta}\d\theta,\delta_0)^p\leq\pi^p.
    \end{equation}
    On the other hand, for $\beta\in[1/2,1)$ we have
    \begin{align*}
            W_p(f_{0,\beta}\d\theta,\delta_0)^p&\leq\frac{1}{\pi}\int_0^\pi\theta^p\frac{1-\beta^2}{1-2\beta\cos\theta+\beta^2}\d\theta=\frac{1}{\pi}\int_0^\pi\theta^p\frac{1-\beta^2}{(1-\beta)^2+2\beta(1-\cos\theta)}\d\theta\leq\\
            &\leq\frac{1}{\pi}\int_0^\pi\theta^p\frac{1-\beta^2}{(1-\beta)^2+\frac{2}{\pi^2}\theta^2}\d\theta\leq\frac{2}{\pi}(1-\beta)\int_0^\pi\frac{\theta^p}{(1-\beta)^2+\frac{2}{\pi^2}\theta^2}\d\theta\leq\\
            &\leq\pi(1-\beta)\int_0^\pi\frac{\theta^p}{(1-\beta)^2+\theta^2}\d\theta
    \end{align*}
    where we used the inequality $$1-\cos\theta\geq\frac{2}{\pi^2}\theta^2,\qquad \forall\theta\in[-\pi,\pi].$$ We now split the last integral as 
    $$I_1\coloneqq\int_0^{1-\beta}\frac{\theta^p}{(1-\beta)^2+\theta^2}\,\d\theta, \qquad I_2\coloneqq\int_{1-\beta}^{\pi}\frac{\theta^p}{(1-\beta)^2+\theta^2}\,\d\theta.$$
    First, for every $p\geq1$,
    $$I_1\leq\int_0^{1-\beta}\frac{\theta^p}{(1-\beta)^2}\,\d\theta =\frac{(1-\beta)^{p-1}}{p+1}.$$
    Moreover, for $p>1$,
    $$I_2\leq\int_{1-\beta}^{\pi}\theta^{p-2}\,\d\theta =\frac{\pi^{p-1}-(1-\beta)^{p-1}}{p-1} \leq\frac{\pi^{p-1}-2^{1-p}}{p-1},$$
    while for $p=1$,
    $$I_2\leq|\ln(1-\beta)|+\ln\pi.$$
    Therefore, for every $\beta\in[1/2,1)$ and every $p>1$,
    $$W_p(f_{0,\beta}\d\theta,\delta_0)^p \leq\hat{C}_p(1-\beta),$$
    where
    $$\hat{C}_p\coloneqq \pi\left(\frac{1}{p+1} + \frac{\pi^{p-1}-2^{1-p}}{p-1}\right).$$
    For $p=1$, we similarly obtain
    $$W_1(f_{0,\beta}\d\theta,\delta_0) \leq \Tilde{C}_1(1-\beta)\bigl(1+|\ln(1-\beta)|\bigr),$$
    where $$\Tilde{C}_1\coloneqq \pi\left(\frac12+\ln\pi\right).$$
    It remains to combine these estimates with the case $\beta\in[0,1/2]$. Since
    $$\frac12\leq1-\beta, \qquad \frac{1+\ln2}{2}\leq(1-\beta)\bigl(1+|\ln(1-\beta)|\bigr),$$
    estimate \eqref{low_beta} implies that, for $p>1$,
    $$W_p(f_{0,\beta}\d\theta,\delta_0)^p \leq2\pi^p(1-\beta),$$
    while for $p=1$,
    $$W_1(f_{0,\beta}\d\theta,\delta_0) \leq \frac{2\pi}{1+\ln2} (1-\beta)\bigl(1+|\ln(1-\beta)|\bigr).$$
    Combining the previous estimates, we conclude that for every $\beta\in[0,1)$ and every $p>1$,
    $$W_p(f_{0,\beta}\d\theta,\delta_0)^p \leq \Tilde{C}_p(1-\beta),$$
    where
    $$\Tilde{C}_p\coloneqq\max\{2\pi^p,\hat{C}_p\}.$$
    For $p=1$, we have
    $$W_1(f_{0,\beta}\d\theta,\delta_0) \leq \Tilde{C}_1(1-\beta)\bigl(1+|\ln(1-\beta)|\bigr).$$
    It remains to derive the exponential decay of $1-\beta(t)$. From \eqref{limit_dyn}, the explicit expression for $\beta(t)$ is
    $$\beta(t)=\frac{\beta(0)} {\sqrt{\beta(0)^2+(1-\beta(0)^2)e^{-Kt}}},$$
    which implies
    $$1-\beta(t) \leq 1-\beta(t)^2 = \frac{(1-\beta(0)^2)e^{-Kt}} {\beta(0)^2+(1-\beta(0)^2)e^{-Kt}} \leq \frac{1-\beta(0)^2}{\beta(0)^2}e^{-Kt}.$$
    Combining this estimate with the monotonicity of $h_1(x)=x$ and $h_2(x)=x(1+|\ln x|)$ completes the proof of the first statement. \jump\jump
    \textbf{ii)} As before we begin by estimating the Wasserstein distance between a generic point of $W^\u(\theta_q^N)$ and the synchronized state $\delta_q$ in terms of $\beta$. We then exploit the dynamics of $\beta_N(t)$ to turn this static estimate into a quantitative convergence result.\jump\jump\noindent
    \fbox{\textbf{Step 1:} Estimate for $W_p\left(\mu^N(\theta(t)),\delta_q\right)$} We fix $A$ and $\beta$ and, for brevity, set $y_i\coloneqq\pi\xi_i+A.$ Moreover, as in Step 1 of the proof of Theorem \ref{main_theorem}, we define
    $$\varphi_i(\beta,A)\coloneqq\phi_i(\alpha,\beta,A)+\alpha.$$
    Then, using again the first integral \eqref{mean_phase_new} and Jensen's inequality, we obtain 
    \begin{align*}
           W_p\left(\mu^N(\theta(t),\delta_{q})\right)^p&\leq\frac{1}{N}\sum_{i=1}^N\Big|\phi_i(\alpha_N(\beta,A),\beta,A)-q\Big|^p=\frac{1}{N}\sum_{i=1}^N\Big|\varphi_i(\beta,A)-\frac{1}{N}\sum_{j=1}^N\varphi_j(\beta,A)\Big|^p\leq\\
           &\leq\frac{2^{p-1}}{N}\sum_{i=1}^N\left(\big|\varphi_i(\beta,A)\big|^p+\frac{1}{N}\sum_{j=1}^N\big|\varphi_j(\beta,A)\big|^p\right)=2^p\frac{1}{N}\sum_{i=1}^N\big|\varphi_i(\beta,A)\big|^p
       \end{align*}
       Let
       $$r(y)\coloneqq d\left(y,\frac{\pi}{2}+\pi\Z\right),\qquad r_i\coloneqq r(y_i),\qquad \rho\coloneqq\frac{1-\beta}{1+\beta}.$$
       We claim that
       $$|\varphi_i|\leq\pi\min\left\{1,\frac{\rho}{r_i}\right\}.$$
       Indeed, by definition one directly has $|\varphi_i|\leq\pi$. On the other hand, using the Lipschitz continuity of $\arctan(\cdot)$, we get
       $$ |\varphi_i| \leq 2\rho|\tan(y_i)| = 2\rho\frac{|\sin(y_i)|}{|\cos(y_i)|} \leq \frac{2\rho}{|\cos(y_i)|} = \frac{2\rho}{\sin(r_i)} \leq \pi\frac{\rho}{r_i}, $$
       where we used that $\sin x\geq\frac{2}{\pi}x$ for every $x\in[0,\pi/2]$. Therefore,
       $$W_p\left(\mu^N(\theta(t)),\delta_q\right)^p \leq \frac{2^p\pi^p}{N} \sum_{i=1}^N \min\left\{1,\frac{\rho^p}{r_i^p}\right\}.$$
       We now derive a lower bound for the quantities $r_i$. Since the points $y_i$ form a uniform grid on $\R/\pi\Z$, with spacing $\pi/N$, we have
       $$ \#\left\{i:r_i\leq R\right\} \leq \frac{2R}{\pi/N}+1. $$
       Up to relabeling, we may assume without loss of generality that
       $$0\leq r_1\leq r_2\leq\cdots\leq r_N.$$
       Thus, for every $m=1,\ldots,N$,
       $$ \#\left\{i:r_i\leq r_m\right\}\geq m. $$
       Combining the last two inequalities gives
       $$r_m\geq\pi\frac{m-1}{2N}.$$
       Using this bound for all terms in the previous summation except the first one, we obtain
       \begin{align*}
        W_p\left(\mu^N(\theta(t),\delta_{q})^p\right)&\leq2^p\pi^p\left(\frac{1}{N}+\frac{1}{N}\sum_{n=2}^N\min\left\{1,\left(\frac{a}{n-1}\right)^p\right\}\right)\leq\\
        &\leq2^p\pi^p\left(\frac{1}{N}+\frac{1}{N}\sum_{n=1}^N\min\left\{1,\left(\frac{a}{n}\right)^p\right\}\right)
       \end{align*}
       with $a\coloneqq\frac{2N\rho}{\pi}$. If $a\in[0,1)$ and $p>1$, then
       $$ \sum_{n=1}^N \min\left\{1,\left(\frac{a}{n}\right)^p\right\} = \sum_{n=1}^N\left(\frac{a}{n}\right)^p \leq a^p\zeta(p) \leq a\zeta(p), $$
       where $\zeta(\cdot)$ denotes the Riemann zeta function. For $p=1$, we use
       $$\sum_{n=1}^N\frac{1}{n}\leq1+\ln N,$$
       and hence
       $$ \sum_{n=1}^N \min\left\{1,\frac{a}{n}\right\} \leq a(1+\ln N). $$
       On the other hand, if $a\geq1$ and $p>1$, then 
       \begin{align*}
            \sum_{n=1}^N\min\left\{1,\left(\frac{a}{n}\right)^p\right\}&=\sum_{n=1}^{\lfloor a\rfloor}\min\left\{1,\left(\frac{a}{n}\right)^p\right\}+\sum_{n=\lfloor a\rfloor+1}^N\min\left\{1,\left(\frac{a}{n}\right)^p\right\}=\\
           &=\lfloor a\rfloor +a^p\sum_{n=\lfloor a\rfloor+1}^N n^{-p}\leq a+a^p\int_{\lfloor a\rfloor}^\infty x^{-p}\d x=\\
           &=a+a^p\frac{\lfloor a \rfloor^{1-p}}{p-1}\leq a+a^p\frac{a^{1-p}2^{p-1}}{p-1}= a\left(1+\frac{2^{p-1}}{p-1}\right).
        \end{align*}
        Similarly, for $p=1$ we obtain
        \begin{align*}
            \sum_{n=1}^N\min\left\{1,\frac{a}{n}\right\}&=\sum_{n=1}^{\lfloor a\rfloor}\min\left\{1,\frac{a}{n}\right\}+\sum_{n=\lfloor a\rfloor+1}^N\min\left\{1,\frac{a}{n}\right\}=\\
            &=\lfloor a\rfloor +a\sum_{n=\lfloor a\rfloor+1}^N \frac{1}{n}\leq a+a\int_{\lfloor a\rfloor}^N \frac{1}{x}\d x=a\left(1+\ln\frac{N}{\lfloor a\rfloor}\right)\leq a\left(1+\ln\frac{2N}{a}\right).
        \end{align*}
        Combining the previous estimates, for $p>1$ we have
        $$ \sum_{n=1}^N \min\left\{1,\left(\frac{a}{n}\right)^p\right\} \leq a\max\left\{\zeta(p),1+\frac{2^{p-1}}{p-1}\right\}. $$
        Since $\rho\leq1-\beta$, this gives 
        $$ \sum_{n=1}^N \min\left\{1,\left(\frac{a}{n}\right)^p\right\} \leq (1-\beta)\frac{2N}{\pi} \max\left\{\zeta(p),1+\frac{2^{p-1}}{p-1}\right\}. $$
        Therefore,
        $$ W_p\left(\mu^N(\theta(t)),\delta_q\right)^p \leq \Tilde{D}_p \left( \frac{1}{N} + 1-\beta \right), $$
        where
        $$ \Tilde{D}_p\coloneqq 2^{1+p}\pi^{p-1} \max\left\{\zeta(p),1+\frac{2^{p-1}}{p-1}\right\}. $$
        In the case $p=1$, the same argument yields
        $$ \sum_{n=1}^N \min\left\{1,\frac{a}{n}\right\} \leq \frac{2N}{\pi} \left(1+\ln(2\pi)\right) (1-\beta) \left( 1+\ln\left(\frac{1}{1-\beta}\right) \right), $$
        and consequently
        $$ W_1\left(\mu^N(\theta(t)),\delta_q\right) \leq \Tilde{D}_1 \left( \frac{1}{N} + (1-\beta) \left( 1+\ln\left(\frac{1}{1-\beta}\right) \right) \right), $$
        with
        $$ \Tilde{D}_1\coloneqq4(1+\ln(2\pi)). $$\jump\noindent \fbox{\textbf{Step 2:} Exponential decay for $\beta_N(t)$} It remains to prove that $\beta_N(t)$ converges exponentially to $1$. From the first equation of \eqref{3d-system} and the estimate $B_N(A,\beta)\geq\frac{N-2}{N}$, obtained in Step 2 of Theorem \ref{main_theorem}, we get $$ \dot{\beta}_N \geq \frac{K}{6}\beta_N(1-\beta_N^2) $$ for all $N\geq3$. Let $x(t)$ be the solution of
        $$ \dot{x} = \frac{K}{6}x(1-x^2), \qquad x(0)=\beta_N(0). $$
        Then
        $$ 1-x(t) \leq \frac{1-x(0)^2}{x(0)^2} e^{-\frac{K}{3}t}. $$
        By the comparison principle, since $\beta_N(0)=x(0)$, we have $\beta_N(t)\geq x(t)$ for all $t\geq0$. Hence,
        $$ 1-\beta_N(t) \leq 1-x(t) \leq \frac{1-\beta_N(0)^2}{\beta_N(0)^2} e^{-\frac{K}{3}t}. $$
        Combining this decay estimate with the bounds obtained in Step 1 proves the second statement of the theorem.
        \end{proof}
        In order to compare the finite-dimensional and limiting reduced dynamics, we first need a quantitative estimate for the corresponding parameters $(\alpha,\beta,A)$ and $(\alpha_N,\beta_N,A_N)$. The following lemma provides a uniform control up to the first time at which either trajectory reaches the level $\Tilde{\beta}$.
        \begin{lemma}\label{estimate_finite_dim_parameters}
        Fix $\Tilde{\beta}\in[0,1)$, $N\geq3$ and let $(\alpha(t),\beta(t),A(t))$ and $(\alpha_N(t),\beta_N(t),A_N(t))$ denote the solutions of \eqref{limit_dyn} and \eqref{3d-system}, respectively, corresponding to the initial data $ (\alpha(0),\beta(0),A(0)),$ $ (\alpha_N(0),\beta_N(0),A_N(0)) \in[-\pi,\pi]\times[0,\Tilde{\beta})\times[-\pi/2,\pi/2].$ If at least one of $\beta(0)$ and $\beta_N(0)$ is non-zero then there exists a constant $\mathcal{C}=\mathcal{C}(\Tilde{\beta},\beta(0),\beta_N(0))>0,$ independent of $N$, such that
        \begin{align*}
             \sup_{0\leq t\leq\tau_N} &\Big( |\beta_N(t)-\beta(t)| + |\alpha_N(t)-\alpha(t)| + |A_N(t)-A(t)| \Big) \leq \\
             &\leq\mathcal{C}\, \frac{\Tilde{\beta}^{N-2}} {(1-\Tilde{\beta}^N)^2}+|\beta_N(0)-\beta(0)|e^{\frac{K}{2}\tau_N}+|\alpha_N(0)-\alpha(0)|+|A_N(0)-A(0)|,
        \end{align*}
        where
        $$ \tau_N\coloneqq \inf\Big\{ t\geq0: \beta(t)=\Tilde{\beta} \ \text{or}\ \beta_N(t)=\Tilde{\beta} \Big\}. $$
        Moreover, $\tau_N$ satisfies
        $$\tau_N\leq\frac{1}{K}\min\left\{\ln\left(\frac{\tilde{\beta}^2(1-\beta(0)^2)}{\beta(0)^2(1-\tilde{\beta}^2)}\right),\frac{4}{3}\ln\left(\frac{\tilde{\beta}^2(1-\beta_N(0)^2)}{\beta_N(0)^2(1-\tilde{\beta}^2)}\right)\right\},$$
        and
        $$\tau_N\geq\frac{1}{K}\min\left\{\ln\left(\frac{\tilde{\beta}^2(1-\beta(0)^2)}{\beta(0)^2(1-\tilde{\beta}^2)}\right),\frac{1}{3}\ln\left(\frac{\tilde{\beta}^2(1-\beta_N(0)^2)}{\beta_N(0)^2(1-\tilde{\beta}^2)}\right)\right\}.$$
        In the case $\beta(0)=\beta_N(0)=0$, we have
        $$\sup_{t\in\R^+}\Big( |\beta_N(t)-\beta(t)| + |\alpha_N(t)-\alpha(t)| + |A_N(t)-A(t)| \Big)\leq|\alpha_N(0)-\alpha(0)|+|A_N(0)-A(0)|.$$
        \end{lemma}
            \begin{proof}
            We now assume that at least one of $\beta(0)$ and $\beta_N(0)$ is non-zero. We begin by estimating the difference between $\beta_N(t)$ and $\beta(t)$. Let $H(x)\coloneqq\frac{K}{2}x(1-x^2).$ Then,
            $$ \dot{\beta}_N-\dot{\beta} = H(\beta_N)B_N(\beta_N,A_N)-H(\beta) = H(\beta_N)-H(\beta) + H(\beta_N)\bigl(B_N(\beta_N,A_N)-1\bigr). $$
            By the definition of $B_N(\cdot)$, we have
            $$ B_N(A_N,\beta_N)-1 = \frac{(1-\beta_N^2)\bigl(\beta_N^{2N-2}+(-\beta_N)^{N-2}\cos(2A_NN)\bigr)} {1+\beta_N^{2N}+2(-\beta_N)^N\cos(2A_NN)} \leq 2\frac{\beta_N^{N-2}}{(1-\beta_N^N)^2} \leq 2\frac{\Tilde{\beta}^{N-2}}{(1-\Tilde{\beta}^N)^2}. $$
            Moreover,
            $$ \sup_{x\in[0,1]}H(x)=K\frac{\sqrt3}{9}, \qquad \sup_{x\in[0,1]}|H'(x)|=\frac{K}{2}. $$
            Therefore,
            $$ |\dot{\beta}_N-\dot{\beta}| \leq \frac{K}{2}|\beta_N-\beta| + K\frac{2\sqrt3}{9} \frac{\Tilde{\beta}^{N-2}}{(1-\Tilde{\beta}^N)^2}. $$
            Applying Gronwall's lemma yields
            $$ |\beta_N(t)-\beta(t)| \leq|\beta_N(0)-\beta(0)|e^{\frac{K}{2}t}+ \frac{4\sqrt3}{9} \left(e^{\frac{K}{2}t}-1\right) \frac{\Tilde{\beta}^{N-2}}{(1-\Tilde{\beta}^N)^2}. $$
            On the other hand,
            $$ |\alpha_N(t)-\alpha(t)| \leq |\alpha_N(0)-\alpha(0)|+\int_0^t \frac{K}{2}\bigl(1+\beta_N^2(s)\bigr) |D_N(A_N(s),\beta_N(s))| \,\d s, $$
            and
            $$ |A_N(t)-A(t)| \leq |A_N(0)-A(0)| + \int_0^t \frac{K}{4}\bigl(1-\beta_N^2(s)\bigr) |D_N(A_N(s),\beta_N(s))| \,\d s. $$
            Furthermore,
            $$ D_N(A_N,\beta_N) = \frac{\sin(2A_NN)(-\beta_N)^{N-2}(\beta_N^2-1)} {1+\beta_N^{2N}+2(-\beta_N)^N\cos(2A_NN)} \leq \frac{\Tilde{\beta}^{N-2}}{(1-\Tilde{\beta}^N)^2}. $$
            Hence,
            $$ |\alpha_N(t)-\alpha(t)| \leq|\alpha_N(0)-\alpha(0)|+ K\tau_N \frac{\Tilde{\beta}^{N-2}}{(1-\Tilde{\beta}^N)^2}, $$
            and
            $$ |A_N(t)-A(t)| \leq|A_N(0)-A(0)|+ \frac{K}{4}\tau_N \frac{\Tilde{\beta}^{N-2}}{(1-\Tilde{\beta}^N)^2}. $$
            Combining the previous estimates completes the proof of our estimate, with
            $$ \mathcal{C}=\mathcal{C}\left(\Tilde{\beta},\beta(0),\beta_N(0)\right) = \frac{5}{4} \tau_N + \frac{4\sqrt3}{9} e^{\frac{K}{2}\tau_N}. $$
            Next, observe that $$\tau_N=\min\left\{\inf\Bigl\{ t\geq0:\beta(t)=\Tilde{\beta} \Bigr\},\;\inf\Bigl\{ t\geq0:\beta_N(t)=\Tilde{\beta} \Bigr\}\right\}.$$
            Using the explicit expression for $\beta(t)$, we obtain 
            $$\inf\Bigl\{ t\geq0:\beta(t)=\Tilde{\beta} \Bigr\}=\frac1K \ln\left( \frac{\Tilde{\beta}^2(1-\beta(0)^2)} {\beta(0)^2(1-\Tilde{\beta}^2)} \right).$$
            On the other hand, a similar argument, together with the fact that for all $N\geq 3$
            $$\frac{1}{3}\leq\frac{N-2}{N}\leq\frac{1-\beta^{N-2}}{1-\beta^N}\leq B_N(A,\beta)\leq\frac{1+\beta^{N-2}}{1+\beta^N}\leq\frac{N+1}{N}\leq\frac{4}{3}$$
            yields
            $$\frac{1}{3K} \ln\left( \frac{\Tilde{\beta}^2(1-\beta(0)^2)} {\beta(0)^2(1-\Tilde{\beta}^2)} \right)\leq\inf\Bigl\{ t\geq0:\beta(t)=\Tilde{\beta} \Bigr\}\leq\frac{4}{3K} \ln\left( \frac{\Tilde{\beta}^2(1-\beta(0)^2)} {\beta(0)^2(1-\Tilde{\beta}^2)} \right),$$
            thereby proving the first part of the theorem. It remains to consider the case $\beta(0)=\beta_N(0)=0$. Since $t\in\R^+$ $\beta(t)=\beta_N(t)=0$, and $D_N(A,0)=0$ for all $A\in[-\pi/2,\pi/2]$, we obtain
            $$|\alpha_N(t)-\alpha(t)|\leq|\alpha_N(0)-\alpha(0)|,\qquad|A_N(t)-A(t)|\leq|A_N(0)-A(0)|$$
            for all $t\in\R$. This completes the proof.
            \end{proof}
Having estimated the difference between the reduced trajectories, we now relate this discrepancy to the corresponding probability densities. The following lemma shows that, on every region of the form $\{\beta\leq\Tilde{\beta}\}$, the map $(\alpha,\beta)\mapsto f_{\alpha,\beta}$ is Lipschitz with respect to the Wasserstein distance.    
\begin{lemma}\label{lips_wrt_wass}
Let $\Tilde{\beta}\in[0,1)$. Then for every $\alpha_1,\alpha_2\in[-\pi,\pi]$, $\beta_1,\beta_2\in[0,\Tilde{\beta}]$, and $p\geq1$, one has
$$ W_p\big(f_{\alpha_1,\beta_1}\d\theta, f_{\alpha_2,\beta_2}\d\theta\big)^p \leq \frac{E_p}{\pi(1-\Tilde{\beta})^2} \Big( |\alpha_2-\alpha_1| + |\beta_2-\beta_1| \Big). $$
\end{lemma}
\begin{proof}
By \cite[Proposition 7.10]{villani2021topics}, for every $\theta_0\in\mathbb{S}^1$,
\begin{align*} W_p\big(f_{\alpha_1,\beta_1}\d\theta, f_{\alpha_2,\beta_2}\d\theta\big)^p &\leq \max\{1,2^{p-1}\} \int_{\mathbb{S}^1} |\theta_0-\theta|^p |f_{\alpha_1,\beta_1}(\theta)-f_{\alpha_2,\beta_2}(\theta)| \d\theta \\
&\leq \pi^p\max\{1,2^{p-1}\} \int_{\mathbb{S}^1} |f_{\alpha_1,\beta_1}(\theta)-f_{\alpha_2,\beta_2}(\theta)| \d\theta .
\end{align*}
A direct computation yields
$$ \partial_\alpha f_{\alpha,\beta}(\theta) = \frac{\beta(\beta^2-1)\sin(\alpha+\theta)} {\pi(1-2\beta\cos(\alpha+\theta)+\beta^2)^2}, $$
and
$$ \partial_\beta f_{\alpha,\beta}(\theta) = \frac{(\beta^2+1)\cos(\alpha+\theta)-2\beta} {\pi(1-2\beta\cos(\alpha+\theta)+\beta^2)^2}. $$
We now derive uniform bounds on these derivatives for $(\theta,\alpha,\beta)\in [0,2\pi]\times[-\pi,\pi]\times[0,\Tilde{\beta}]$. First, observe that
$$ |(\beta^2+1)\cos(\alpha+\theta)-2\beta| \leq 1-2\beta\cos(\alpha+\theta)+\beta^2. $$
Indeed,
$$ \bigl(1-2\beta\cos(\alpha+\theta)+\beta^2\bigr) -\bigl((\beta^2+1)\cos(\alpha+\theta)-2\beta\bigr) =(1+\beta)^2(1-\cos(\alpha+\theta)) \geq0, $$
while
$$ \bigl(1-2\beta\cos(\alpha+\theta)+\beta^2\bigr) +\bigl((\beta^2+1)\cos(\alpha+\theta)-2\beta\bigr) =(1-\beta)^2(1+\cos(\alpha+\theta)) \geq0. $$
Consequently,
$$ |\partial_\beta f_{\alpha,\beta}(\theta)| \leq \frac{1} {\pi\bigl(1-2\beta\cos(\alpha+\theta)+\beta^2\bigr)} \leq \frac{1}{\pi(1-\beta)^2} \leq \frac{1}{\pi(1-\Tilde{\beta})^2}. $$
On the other hand, rewriting $\partial_\alpha f_{\alpha,\beta}$ as
$$ \partial_\alpha f_{\alpha,\beta}(\theta) = \frac{\sqrt{\beta}(1+\beta)}{\pi} \frac{\sqrt{\beta}(\beta-1)\sin(\alpha+\theta)} {\Bigl((1-\beta)^2+ \bigl(2\sqrt{\beta} \sin\bigl(\frac{\alpha+\theta}{2}\bigr)\bigr)^2 \Bigr)^2}, $$
we obtain
\begin{align*}
        |\partial_\alpha f_{\alpha,\beta}(\theta)|&\leq\frac{2}{\pi}\frac{\sqrt{\beta}|\sin(\alpha+\theta)|(1-\beta)}{\left((1-\beta)^2+\left(2\sqrt{\beta}\sin\left(\frac{\alpha+\theta}{2}\right)\right)^2\right)^2}\leq\frac{2}{\pi}\frac{2\sqrt{\beta}\left|\sin\left(\frac{\alpha+\theta}{2}\right)\right|(1-\beta)}{\left((1-\beta)^2+\left(2\sqrt{\beta}\sin\left(\frac{\alpha+\theta}{2}\right)\right)^2\right)^2}\leq\\
        &\leq\frac{1}{\pi}\frac{\left(2\sqrt{\beta}\left|\sin\left(\frac{\alpha+\theta}{2}\right)\right|\right)^2+(1-\beta)^2}{\left((1-\beta)^2+\left(2\sqrt{\beta}\sin\left(\frac{\alpha+\theta}{2}\right)\right)^2\right)^2}\leq\frac{1}{\pi\left(1-\beta^\star\right)^2}
    \end{align*}
The conclusion now follows from the mean value theorem.
\end{proof}
We are now in a position to combine the previous results and establish a uniform-in-time mean-field limit for trajectories of \eqref{KMH} lying on $W^u(\theta_q^N)$.
\begin{theorem}\label{unif_MFL}
Fix $\left(\alpha(0),\beta(0)\right)\in[-\pi,\pi]\times(0,1)$ and $p\geq1$, then for every sequence of initial data $\left(\alpha_N(0),\beta_N(0)\right)_{N\in\N}\in[-\pi,\pi]^\N\times[0,1)^\N$ such that
$$\lim_{N\rightarrow+\infty}|\alpha_N(0)-\alpha(0)|=0\qquad\text{and}\qquad\lim_{N\rightarrow+\infty}|\beta_N(0)-\beta(0)|=0$$
we have that
\begin{equation}\label{uniform_MFL_eq}
    \lim_{N\rightarrow+\infty}\sup_{t\in\R^+}W_p\left(\mu^N\left(\theta(\alpha_N(t),\beta_N(t),A_N(t))\right),f_{\alpha(t),\beta(t)}\right)=0,
\end{equation}
with $\theta(\alpha_N(t),\beta_N(t),A_N(t))\in W^u\left(\theta_{\alpha(0)}^N\right)$. Furthermore, if $\beta(0)=0$, then for every sequence of times $(T_N)_{N\in\N}$ satisfying
$$\lim_{N\rightarrow+\infty}T_N=+\infty\qquad\text{and}\qquad\lim_{N\rightarrow+\infty}\beta_N(0)e^{KT_N}=0$$
it holds that
\begin{equation}\label{long-time_behaviour}
    \lim_{N\rightarrow+\infty}\sup_{t\in[0,T_N]}W_p\left(\mu^N\left(\theta(\alpha_N(t),\beta_N(t),A_N(t))\right),f_{\alpha(t),\beta(t)}\right)=0.
\end{equation}
\end{theorem}
\begin{proof}
    We first consider the case $\beta(0)\in(0,1)$. To this end, we define the \textit{splitting} time as in Lemma \ref{estimate_finite_dim_parameters}
    $$\tau_N\coloneqq \inf\Big\{ t\geq0: \beta(t)=\Tilde{\beta}(N) \ \text{or}\ \beta_N(t)=\Tilde{\beta}(N) \Big\}$$
    with $\tilde{\beta}(N)=1-\delta_N\in[0,1)$ for every $N\in\mathbb{N}$, where the sequence $(\delta_N)_{N\in\mathbb{N}}$ satisfies
    $$\lim_{N\rightarrow+\infty}\delta_N=0, \qquad\lim_{N\rightarrow+\infty}\delta_NN=+\infty,\qquad\max\{\beta(0),\beta_N(0)\}<\tilde{\beta}(N) \quad\forall N\in\N,$$
    $$\lim_{N\rightarrow+\infty}\frac{|\alpha_N(0)-\alpha(0)|}{\delta_N^2}=0, \qquad\lim_{N\rightarrow+\infty}\frac{|\beta_N(0)-\beta(0)|}{\delta_N^{\frac{5}{2}}}=0.$$
    As already anticipated, the main idea of the proof is to estimate differently the irregular time regime $t\geq \tau_N$, for which $\beta\approx1$, and the regular one $t\leq \tau_N$ when $\beta$ is distant enough from 1. For the sake of clarity in the following we denote $\theta(\alpha_N(t),\beta_N(t),A_N(t))$ simply with $\theta^N(t)$. With this in mind we now define
$$\mathrm{I}_N\coloneqq\sup_{t\in[0,\tau_N]}W_p\left(\mu^N\left(\theta^N(t)\right),f_{\alpha(t),\beta(t)}\right)^p,$$
$$\mathrm{II}_N\coloneqq\sup_{t\in[\tau_N,+\infty)}W_p\left(\mu^N\left(\theta^N(t)\right),f_{\alpha(t),\beta(t)}\right)^p$$
    so that $$\sup_{t\in\R^+}W_p\left(\mu^N\left(\theta^N(t)\right),f_{\alpha(t),\beta(t)}\right)^p=\max\left\{\mathrm{I}_N,\mathrm{II}_N\right\}.$$
    From the definition of $\tau_N$ and Lemma \ref{estimate_finite_dim_parameters} we know that
    $$\tau_N\geq\frac{1}{3K}\ln\left(\frac{\tilde{\beta}(N)^2\left(1-\max\{\beta(0),\beta_N(0)\}^2\right)}{\max\{\beta(0),\beta_N(0)\}^2(1-\tilde{\beta}(N)^2)}\right)\eqqcolon\tau_N^-.$$
    In the following we prove the statement of theorem for $p>1$, but a similar argument holds for $p=1$. We can easily estimate $\mathrm{II}_N$ by using Theorem \ref{conv_rate_synchro}, indeed 
     \begin{align*}
2^{1-p}\mathrm{II}_N\leq&\sup_{t\in[\tau_N,+\infty)}W_p\left(\mu^N\left(\theta^N(t)\right),\delta_q\right)^p+\sup_{t\in[\tau_N,+\infty)}W_p\left(\delta_q,f_{\alpha(t),\beta(t)}\right)^p\leq\\
\leq&\sup_{t\in[\tau_N,+\infty)}\Tilde{C}_p\frac{1-\beta_N(0)^2}{\beta_N(0)^2}e^{-Kt}+\Tilde{D}_p\left(\frac{1}{N}+\frac{1-\beta(0)^2}{\beta(0)^2}e^{-\frac{K}{3}t}\right)=\\
=&\Tilde{C}_p\frac{1-\beta_N(0)^2}{\beta_N(0)^2}e^{-K\tau_N^-}+\Tilde{D}_p\left(\frac{1}{N}+\frac{1-\beta(0)^2}{\beta(0)^2}e^{-\frac{K}{3}\tau_N^-}\right)=\\
=&O\left(\delta_N^{\frac{1}{3}}+\frac{1}{N}+\delta_N^{\frac{1}{9}}\right)=O\left(\frac{1}{N}+\delta_N^{\frac{1}{9}}\right).
    \end{align*}
    The estimate of $\mathrm{I}_N$ proceed as follows
    \begin{align*}
       \frac{\mathrm{I}_N}{2^{p-1}}\leq& \sup_{t\in[0,\tau_N]}W_p\left(\mu^N\left(\theta^N(t)\right),f_{\alpha_N(t),\beta_N(t)}\right)^p+\sup_{t\in[0,\tau_N]}W_p\left(f_{\alpha_N(t),\beta_N(t)},f_{\alpha(t),\beta(t)}\right)^p\leq\\
       \leq&\frac{C_p(\tilde{\beta}(N))^p}{N^p}+\frac{E_p}{(1-\tilde{\beta}(N))^2}\sup_{t\in[0,\tau_N]}\bigg(|\alpha_N(t)-\alpha(t)|+|\beta_N(t)-\beta(t)|\bigg)
    \end{align*}
    where for the first addend we have used Theorem \ref{conv_manifolds} and for the second addend Lemma \ref{lips_wrt_wass}.
    We continue by applying Lemma \ref{estimate_finite_dim_parameters} to the second addend, therefore obtaining that
    \begin{align*}
        \frac{\mathrm{I}_N}{2^{p-1}}\leq&\frac{C_p(\tilde{\beta}(N))^p}{N^p}+\frac{E_p}{(1-\tilde{\beta}(N))^2}\left(\frac{\mathcal{C}(\tilde{\beta}(N)){\tilde{\beta}(N)}^{N-2}}{\left(1-{\tilde{\beta}(N)}^N\right)^2}+|\beta_N(0)-\beta(0)|\sqrt{\frac{\tilde{\beta}(N)^2(1-\beta(0)^2)}{(1-\tilde{\beta}(N)^2)\beta(0)^2}}+\right.\\
        &+|\alpha_N(0)-\alpha(0)|+\left|A_N(0)-\frac{2\alpha(0)+\pi}{2}\right| \biggr)\leq\\
        \leq&\frac{C_p^p}{\left(1-\tilde{\beta}(N)\right)^p}\frac{1}{N^p}+\frac{E_p\mathcal{C}}{(1-\tilde{\beta}(N))^2}\frac{{\tilde{\beta}(N)}^{N-2}}{\left(1-{\tilde{\beta}(N)}^N\right)^2}\left(\ln\left(\frac{\tilde{\beta}(N)^2(1-\beta(0)^2)}{\beta(0)^2(1-\tilde{\beta}(N)^2)}\right)+\right.\\
        +&\left.\sqrt{\frac{\tilde{\beta}(N)^2(1-\beta(0)^2)}{\beta(0)^2(1-\tilde{\beta}(N)^2)}}\right)+\frac{E_p\mathcal{C}}{(1-\tilde{\beta}(N))^2}\Bigg(|\beta_N(0)-\beta(0)|\sqrt{\frac{\tilde{\beta}(N)^2(1-\beta(0)^2)}{(1-\tilde{\beta}(N)^2)\beta(0)^2}}+\\
        +&|\alpha_N(0)-\alpha(0)|+\left|A_N(0)-\frac{2\alpha(0)+\pi}{2}\right|\Bigg)=\\
        =&O\left(\frac{1}{\left(\delta_NN\right)^p}+\frac{e^{-\delta_NN}}{\delta_N^2}\left[1-\ln\left(\delta_N\right)+\frac{1}{\sqrt{\delta_N}}\right]+\frac{|\beta_N(0)-\beta(0)|}{\delta_N^{\frac{5}{2}}}+\frac{|\alpha_N(0)-\alpha(0)|}{\delta_N^{2}}\right).
    \end{align*}
    The contribution of $A_N(0)$ disappears in the limit $N\to\infty$, due to its explicit dependence on $(\alpha_N(0),\beta_N(0))$ established in Lemma \ref{first_int}.
    From the estimates on $\mathrm{I}_N$ and $\mathrm{II}_N$ we can therefore already deduce the validity of \eqref{uniform_MFL_eq}.\\

    We now turn to the case $\beta(0)=0$. In this case, by \eqref{3d-system} and the estimate
    $$B_N(A,\beta)\leq\frac{1+\beta^{N-2}}{1+\beta^N}\leq2$$
    we obtain
    $$\dot{\beta}_N(t)\leq K\beta_N(t).$$
    An application of Gronwall's inequality then yields
    $$\beta_N(t)\leq\beta_N(0)e^{Kt}.$$
    Using the assumption on $T_N$, we deduce that
    $$\lim_{N\rightarrow+\infty}\sup_{t\in[0,T_N]}|\beta_N(t)-\beta(t)|=\lim_{N\rightarrow+\infty}\sup_{t\in[0,T_N]}\beta_N(t)=0.$$
    Therefore, for $N$ sufficiently large,
    \begin{equation}\label{unif_est_beta}
        \sup_{t\in[0,T_N]}\beta_N(t)\leq\frac{1}{2}.
    \end{equation}
    Proceeding as in the proof of $\ref{estimate_finite_dim_parameters}$, we obtain
    \begin{align*}
        \sup_{t\in[0,T_N]}|\alpha_N(t)-\alpha(t)|\leq&|\alpha_N(0)-\alpha(0)|+4K\int_0^{T_N}\beta_N(s)\d s\leq\\
        \leq&|\alpha_N(0)-\alpha(0)|+4\beta_N(0)e^{KT_N}.
    \end{align*}
    Applying an estimate analogous to that used for $\mathrm{I}_N$, together with \eqref{unif_est_beta}, completes the proof.
\end{proof}
It is worth noting that, when $\beta(0)=0$, uniform-in-time convergence cannot, in general, be expected for an arbitrary sequence $\left(\beta_N(0)\right)_{N\in\N}.$ In this case, one can only obtain a long-time convergence result such as \eqref{long-time_behaviour}. Indeed, if $\beta_N(0)>0$ for every $N$, then \eqref{3d-system} implies that $$ \lim_{t\to+\infty}\beta_N(t)=1 $$ for every $N$, whereas $\beta(t)=0$ for all $t\in\mathbb{R}_+$. Consequently, \eqref{uniform_MFL_eq} cannot hold. The next corollary provides sufficient conditions under which uniform-in-time convergence is recovered in the case $\beta(0)=0$. It also yields an explicit convergence rate of \eqref{uniform_MFL_eq} for every $\beta(0)\in[0,1)$.
\begin{corollary}
In the setting of the previous theorem, when $\beta(0)\in(0,1)$ if we have that
\begin{equation}\label{O_parameters}
    |\alpha_N(0)-\alpha(0)|=O\left(N^{-\frac{19p}{1+9p}}\right)\qquad\text{and}\qquad|\beta_N(0)-\beta(0)|=O\left(N^{-\frac{47p}{2+18p}}\right)
\end{equation}
then there exists a positive constant $\frak{C}_p>0$ such that  we have that for every $p>1$ it holds
\begin{equation}\label{W_p_rate}
    \sup_{t\in\R^+}W_p\left(\mu^N\left(\theta(\alpha_N(t),\beta_N(t),A_N(t))\right),f_{\alpha(t),\beta(t)}\right)^p\leq \frak{C}_pN^{-\frac{p}{1+9p}}
\end{equation}
and
\begin{equation}\label{W_1_rate}
    \sup_{t\in\R^+}W_1\left(\mu^N\left(\theta(\alpha_N(t),\beta_N(t),A_N(t))\right),f_{\alpha(t),\beta(t)}\right)\leq \frak{C}_1N^{-\frac{1}{10}}\left(\ln N\right)^{\frac{9}{10}}
\end{equation}
therefore attaining the best convergence rate obtainable with our technical estimates.
Moreover, if $\beta(0)=0$ and $\beta_N(0)=0$ for every $N\in\N$, then we obtain
$$\sup_{t\in\R^+}W_p\left(\mu^N\left(\theta(\alpha_N(t),\beta_N(t),A_N(t))\right),f_{\alpha(t),\beta(t)}\right)\leq \frac{C_p}{N}.$$
\end{corollary}
\begin{proof}
From the proof of Theorem \ref{unif_MFL}, if $\beta(0)\in(0,1)$, then there exists a positive constant $\frak{C}_p>0$, which for brevity we allow to change its value from line to line, such that
$$\sup_{t\in\R^+}W_p\left(\mu^N\left(\theta^N(t)\right),f_{\alpha(t),\beta(t)}\right)^p\leq \frak{C}_pR_N$$
    with
    $$R_N\coloneqq\max\left\{\frac{1}{\left(\delta_NN\right)^p}+\frac{e^{-\delta_NN}}{\delta_N^3}+\frac{|\beta_N(0)-\beta(0)|}{\delta_N^{\frac{5}{2}}}+\frac{|\alpha_N(0)-\alpha(0)|}{\delta_N^{2}},\frac{1}{N}+\delta_N^{\frac{1}{9}}\right\}.$$
        We now choose $\delta_N$, and consequently $\tilde{\beta}(N)$, in order to get the best rate of convergence for $R_N$. We start by noticing that
    $$R_N\geq\max\left\{\frac{1}{\left(\delta_NN\right)^p},\delta_N^{\frac{1}{9}}\right\}\geq\left(\frac{1}{\left(\delta_NN\right)^p}\right)^{\frac{\frac{1}{9}}{\frac{1}{9}+p}}\left(\delta_N^{\frac{1}{9}}\right)^{\frac{p}{\frac{1}{9}+p}}=N^{-\frac{p}{1+9p}}$$
    implying that we can not, with our estimates, obtain a better rate of convergence than $N^{-\frac{p}{1+9p}}$. A straightforward computation shows that by choosing
    $$\delta_N=O\left(N^{-\frac{9p}{1+9p}}\right)$$
    and assuming \eqref{O_parameters} we exactly attain such a minimal rate of convergence, proving the theorem. Similar computations for $p=1$ leads to 
$$\sup_{t\in\R^+}W_1\left(\mu^N\left(\theta^N(t)\right),f_{\alpha(t),\beta(t)}\right)^p\leq \frak{C}_1N^{-\frac{1}{10}}\left(\ln N\right)^{\frac{9}{10}}.$$
On the other hand, suppose that $\beta(0)=0$ and $\beta_N(0)=0$ for every $N\in\N$. Then $\beta(t)=0$ and $\beta_N(t)=0$ for all $t\in\R^+$ and every $N\in\N$. Hence,
\begin{align*}
    W_p\left(\mu^N\left(\theta^N(t)\right),f_{\alpha(t),\beta(t)}\right)\leq& W_p\left(\mu^N\left(\theta^N(t)\right),f_{\alpha_N(t),\beta_N(t)}\right)+W_p\left(f_{\alpha_N(t),\beta_N(t)},f_{\alpha(t),\beta(t)}\right)=\\
    =&W_p\left(\mu^N\left(\theta^N(t)\right),f_{\alpha_N(t),0}\right)+W_p\left(f_{\alpha_N(t),0},f_{\alpha(t),0}\right)\leq\\
    \leq&\frac{C_p}{N}+0,
\end{align*}
where we used Theorem~\ref{conv_manifolds} and the fact that
$$f_{\alpha,0}=\frac{\d\theta}{2\pi}$$for every $\alpha\in[-\pi,\pi]$. This concludes the proof.
\end{proof}

We would like to end this section by observing how the estimates \eqref{W_p_rate} and \eqref{W_1_rate} are related. Given that for every $\varepsilon>0$, for every $N>1$ the following basic estimate holds
$$\ln N\leq\frac{N^\varepsilon}{e\varepsilon}$$
we can easily deduce that for every $\varepsilon>0$
$$\sup_{t\in\R^+}W_1\left(\mu^N\left(\theta^N(t)\right),f_{\alpha(t),\beta(t)}\right)\leq \frac{\frak{C}_1}{e\varepsilon}N^{-\frac{2}{11}+\varepsilon}$$
which resembles \eqref{W_p_rate} in the limit $p\rightarrow1$. Moreover we know from \cite[Chapter 7]{villani2021topics} that for any $p_1\leq p_2$ it holds
$$W_{p_1}\leq W_{p_2}$$
which shows how the previous estimate is indeed compatible once again with \eqref{W_p_rate}, since
$$\sup_{t\in\R^+}W_1\left(\mu^N\left(\theta^N(t)\right),f_{\alpha(t),\beta(t)}\right)\leq\sup_{t\in\R^+}W_p\left(\mu^N\left(\theta^N(t)\right),f_{\alpha(t),\beta(t)}\right)\leq\frak{C}_p^{\frac{1}{p}}N^{-\frac{2}{2+9p}}.$$
We furthermore emphasize that since we are working on a compact domain, with $\operatorname{diam}\left(\mathbb{S}^1\right)=\pi$, then we have that all the Wasserstein distances $W_p$ are equivalent since for any $p_1\geq p_2$ it holds
$$W_{p_1}\leq\pi^{1-\frac{p_2}{p_1}} W_{p_2}^{\frac{p_2}{p_1}},$$
therefore implying that the uniform in time mean-field limit for all $p\geq1$ takes place with the best convergence rate among all the $p\geq1$.
\end{section}





%% file: conclusion.tex
\begin{section}{Conclusion and Outlook}\label{concl}
In this work, we have studied the relation between an unstable manifold of the finite-dimensional Kuramoto model and the Ott–Antonsen manifold of its mean-field limit. We first obtained a complete spectral characterization of the equilibria of the homogeneous Kuramoto model and used it to construct an explicit parametrization of the global unstable manifold associated with the incoherent equilibria for $N\geq3$. The resulting parametrization is particularly useful for studying the limit $N\to\infty$, since the dynamics on this manifold can be directly related to the reduced dynamics on the Ott–Antonsen manifold.\\

The main result establishes the convergence of the finite-dimensional unstable manifolds to their continuum counterparts. More precisely, after applying suitable empirical-measure and step-function lifts, we obtain convergence in Hausdorff distance to the Ott–Antonsen manifold and to the corresponding unstable manifold of the continuum limit. On every parameter range $\beta\in[0,\Tilde{\beta}]$, with $\Tilde{\beta}<1$, the convergence is of order $1/N$. This provides, in particular, a geometric connection between invariant manifolds of the particle system and invariant manifolds of its mean-field limit.\\

The results obtained here suggest that the convergence of invariant structures may provide a useful perspective on mean-field limits beyond the convergence of individual trajectories. In particular, they indicate the existence of regions of phase space in which the mean-field approximation holds with improved, and even uniform-in-time, convergence rates. A natural direction for future research is to determine how robust and general this phenomenon is. More broadly, it would be interesting to develop a general theory for the convergence of invariant manifolds associated with finite-dimensional particle systems towards their counterparts, whenever they exist, in the mean-field limit. The results of the present work rely on the explicit parametrization of the underlying finite-dimensional manifolds. It would therefore be particularly desirable to establish such convergence results within a more abstract framework, without requiring explicit descriptions of the invariant structures involved.
\subsection*{Acknowledgements}
 The authors have been supported by the  Deutsche Forschungsgemeinschaft (DFG, German Research Foundation) - Project ID 543163250.
\end{section}

%% file: Appendix.tex
\appendix\label{APP}

\begin{section}{Complementary formulas}\label{App_A}
For the sake of completeness, in this appendix we prove a couple of identities that are used throughout the proofs of our main results. We recall that $\mathrm{Arg}(z) : \mathbb{C} \setminus \{0\} \to (-\pi, \pi]$, with $z=x+\i y$, is defined as
$$\mathrm{Arg}(z) \coloneqq
\begin{cases}
\arctan\left(\frac{y}{x}\right) & \text{if } x > 0,\; y \in \mathbb{R}, \\
\arctan\left(\frac{y}{x}\right) + \pi & \text{if } x < 0,\; y \ge 0, \\
\arctan\left(\frac{y}{x}\right) - \pi & \text{if } x < 0,\; y < 0, \\
\frac{\pi}{2} & \text{if } x = 0,\; y > 0, \\
-\frac{\pi}{2} & \text{if } x = 0,\; y < 0, \\
\text{undefined} & \text{if } x = 0,\; y = 0.
\end{cases}$$

\begin{lemma}
    The following trigonometric identities hold. For each $\beta\in[0,1)$ and $\alpha\in[-\pi,\pi]$, we have
        \begin{equation}\label{trig_1}
        \Arg\left(\frac{1-\beta e^{-\i\alpha}}{1-\beta e^{\i\alpha}}\right)=2\arctan\left(\frac{\beta\sin(\alpha)}{1-\beta\cos(\alpha)}\right)
\end{equation}
\begin{equation}\label{trig_2}
\arctan\left(\frac{1+\beta}{1-\beta}\tan\left(\frac{\alpha}{2}\right)\right)-\arctan\left(\frac{\beta\sin(\alpha)}{1-\beta\cos(\alpha)}\right)=\frac{\alpha}{2}
\end{equation}
\end{lemma}

\begin{proof}
We start by setting
$$z\coloneqq 1-\beta e^{\i\alpha}=1-\beta\cos(\alpha)-\i\beta\sin(\alpha).$$
Since $\text{Re}(z)>0$, it follows that
$$\Arg(z)=-\arctan\left(\frac{\beta\sin(\alpha)}{1-\beta\cos(\alpha)}\right).$$
Using standard properties of the argument function, we obtain
$$\Arg\left(\frac{1-\beta e^{-\i\alpha}}{1-\beta e^{\i\alpha}}\right)
=\Arg\left(\frac{\overline{z}}{z}\right)
=\Arg(\overline{z})-\Arg(z)
=-2\Arg(z)
=2\arctan\left(\frac{\beta\sin(\alpha)}{1-\beta\cos(\alpha)}\right),$$
which proves \eqref{trig_1}. In order to prove \eqref{trig_2}, we proceed similarly. Indeed we write
$$z=e^{\i\alpha/2}\left(e^{-\i\alpha/2}-\beta e^{\i\alpha/2}\right)
=e^{\i\alpha/2}\Big((1-\beta)\cos(\alpha/2)-\i (1+\beta)\sin(\alpha/2)\Big)$$
and applying $\Arg(\cdot)$ to both sides yields
$$-\arctan\left(\frac{\beta\sin(\alpha)}{1-\beta\cos(\alpha)}\right)
=\Arg(z)
=\frac{\alpha}{2}-\arctan\left(\frac{(1+\beta)\sin(\alpha/2)}{(1-\beta)\cos(\alpha/2)}\right),$$
which proves \eqref{trig_2}.
\end{proof}

We now derive an explicit expression of $A$ as a function of $\alpha$ and $\beta$. The idea is essentially to use again that the mean phase of the system is conserved as in \eqref{mean_phase_new}, where $\alpha$ was easily expressed as a function of $A$ and $\beta$. We are going to denote such function as $A_N(\alpha,\beta)$ or just $A(\alpha,\beta)$ or $A$ when $N$ is clear from the context, depending on what we want to emphasize.
\begin{lemma}\label{first_int}
Fix a mean phase $q\in\T$, then in the $(\alpha,\beta,A)$ coordinates we have that
\begin{equation}\label{first_int_express}
    \Lambda\coloneqq e^{\i Nq}=e^{-\i N\alpha}\frac{\beta^N+(-1)^Ne^{2\i NA}}{1+(-\beta)^Ne^{2\i N A}}
\end{equation}
is a conserved quantity as well. Moreover we obtain that
\begin{equation}\label{A_express}
A_N(\alpha,\beta)=\frac{\alpha+q+\pi}{2}+\frac{1}{N}\arctan\left(\frac{\beta^N\sin\left(N(\alpha+q)\right)}{1-\beta^N\cos\left(N(\alpha+q)\right)}\right)\qquad \operatorname{mod}\;\frac{\pi}{N}
\end{equation}

\end{lemma}
\begin{proof}
By using that with our choice of $\phi_i$ we have that
$$e^{\i \phi_i}=e^{-\i \alpha}\frac{C\eta_i+\beta}{1+C\beta \eta_i}$$
with $C=e^{2\i A}$ and $\eta_i=e^{2\pi\i\xi_i}$ as in \eqref{exp_expression}. From this it follows that

$$\Lambda=\prod_{j=1}^N e^{\i\phi_j(t)}=e^{-\i N\alpha}\frac{\prod_{i=1}^N(C\eta_i+\beta)}{\prod_{i=1}^N(1+\beta C\eta_i)}$$
but we know that
$$z^N+C^N=\prod_{i=1}^N(z-C\eta_i)\qquad\forall z\in\C$$
so that computing the latter expression for $z=-\beta$ we obtain
$$\prod_{i=1}^N(C\eta_i+\beta)=\beta^N+(-1)^NC^N$$
and doing the same with $z=-\beta^{-1}$ yields
$$\prod_{i=1}^N(1+\beta C\eta_i)=1+(-\beta)^NC^N$$
which proves \eqref{first_int_express}. From this, and the definition of $\Lambda$ we obtain that 
$$e^{\i Nq}=e^{-\i N\alpha}\frac{\beta^N+(-1)^Ne^{2\i NA}}{1+(-\beta)^Ne^{2\i N A}}$$
which implies that
$$e^{2\i NA}=(-1)^N\frac{e^{\i N(\alpha+q)}-\beta^N}{1-\beta^N e^{\i N(\alpha+q)}}=e^{\i N(\alpha+q+\pi)}\frac{1-\beta^Ne^{-\i N(\alpha+q)}}{1-\beta^N e^{\i N(\alpha+q)}}.$$
By applying \eqref{trig_1} we obtain that
$$2NA=N(\alpha+q+\pi)+2\arctan\left(\frac{\beta^N\sin(N(\alpha+q))}{1-\beta^N\cos(N(\alpha+q))}\right)\qquad\operatorname{mod}\;2\pi$$
concluding our proof.
\end{proof}

\begin{lemma}\label{chain_rule}
For every $\alpha\in[-\pi,\pi]$, $\beta\in[0,1)$, and $A\in[-\pi/2,\pi/2]$, we have
\begin{equation}\label{der_alpha}
    \partial_\alpha\phi_i(\alpha,\beta,A)=-1,
\end{equation}
\begin{equation}\label{der_beta}
\partial_\beta\phi_i(\alpha,\beta,A)=-2\frac{\sin(\phi_i+\alpha)}{1-\beta^2},
\end{equation}
\begin{equation}\label{der_A}
    \partial_A\phi_i(\alpha,\beta,A)=2\frac{1-2\beta\cos(\phi_i+\alpha)+\beta^2}{1-\beta^2},
\end{equation}
where $\phi_i(\alpha,\beta,A)$ is defined as in \eqref{fin_N_ansatz}.
\end{lemma}

\begin{proof}
Equation \eqref{der_alpha} is immediate, therefore we focus on \eqref{der_beta}. Define
$$u_i\coloneqq\pi\xi_i+A,\qquad r\coloneqq\frac{1-\beta}{1+\beta},\qquad x\coloneqq r\tan u_i.$$
Then
$$\partial_\beta\phi_i
=2\frac{1}{1+x^2}\,\partial_\beta(r\tan u_i)
=-4\frac{\tan u_i}{(1+\beta)^2(1+x^2)}
=-\frac{4x}{(1-\beta^2)(1+x^2)}.$$
Using the double-angle formula, we obtain
$$\sin(\phi_i+\alpha)=\sin(2\arctan x)=\frac{2x}{1+x^2},$$
which yields \eqref{der_beta}. For the derivative with respect to $A$, we have
$$\partial_A \phi_i=2\frac{1}{1+x^2}r\sec^2 u_i.$$
Again using trigonometric identities, we find
$$\cos(\phi_i+\alpha)=\frac{1-x^2}{1+x^2}$$
and by substituting this into \eqref{der_A}, we obtain
$$2\frac{1-2\beta\cos(\phi_i+\alpha)+\beta^2}{1-\beta^2}
=2\frac{(1-\beta)^2+(1+\beta)^2x^2}{(1-\beta^2)(1+x^2)}.$$
In the end, from using the definition of $x$, it follows that
$$(1-\beta)^2+(1+\beta)^2x^2=(1-\beta)^2(1+\tan^2 u_i)=(1-\beta)^2\sec^2 u_i,$$
which proves \eqref{der_A}.
\end{proof}
\end{section}

\begin{section}{Alternative computation}\label{App_B}
In this part of the appendix, we provide an alternative approach to the one used in Proposition~\ref{prop_spectral}, which allows us to derive the eigenvalues and eigenvectors of the incoherent equilibrium by means of a different method.  In what follows, without loss of generality, we consider the incoherent equilibrium with $q=0$. We denote by
$$c_i\coloneqq\cos(\theta^\star_i)=\cos\left(\frac{2\pi i}{N}\right),\qquad 
s_i\coloneqq\sin(\theta^\star_i)=\sin\left(\frac{2\pi i}{N}\right)$$
the $i$-th entries of the vectors $\underline{c}$ and $\underline{s}$, respectively.

\begin{lemma}
    The incoherent equilibrium $\theta^\star$ has 
    \begin{itemize}
        \item $\lambda_1=0$ as an eigenvalue, with associated eigenspace given by $(\operatorname{span}\{\underline{c},\underline{s}\})^\perp$;
        \item $\lambda_2=\frac{K}{2}$ as an eigenvalue, with associated eigenspace given by $\operatorname{span}\{\underline{c},\underline{s}\}$.
    \end{itemize}
\end{lemma}

\begin{proof}
    Let $\theta^\star$ be an incoherent equilibrium. Then
$$\sum_{j=1}^Ne^{\i\theta^\star_j}=0,$$
which implies
$$\sum_{j=1}^N \cos{\theta^\star_j}=0.$$
Therefore, the entries of the Jacobian matrix \eqref{gen_jacobian} simplify to 
$$DV(\theta^\star)_{ij}=\frac{K}{N}\cos\left(\frac{2\pi}{N}(j-i)\right).$$
We now exploit the fact that this is a circulant matrix, which allows us to compute its eigenvalues explicitly; see \cite{tee2005eigenvectors}. Denoting the $N$-th roots of unity by
$$\rho_k\coloneqq e^{\i\frac{2\pi}{N}k},$$
it can be verified that $DV(\theta^\star)$ has $N$ orthogonal eigenvectors and eigenvalues given by
$$
\mathbf{w}^{(k)}=\left(\begin{array}{l}
1 \\
\rho_k \\
\rho_k^2 \\
\vdots \\
\rho_k^{N-1}
\end{array}\right), \quad 
\lambda_k=J_{1,1}+J_{1,2} \rho_k+J_{1,3}\rho_k^2+\cdots+J_{1,N} \rho_k^{N-1},
$$
for $k=0,1,2,\ldots,N-1$. By using the explicit form of the entries $DV(\theta^\star)_{ij}$, we obtain
\begin{align*}
    \frac{N}{K}\lambda_k=&\sum_{m=0}^{N-1} \cos \left(\frac{2 \pi}{N} m\right) e^{2 \pi i k m / N}=\sum_{m=0}^{N-1} \frac{e^{2 \pi i m / N}+e^{-2 \pi i m / N}}{2} e^{2 \pi i k m / N}=\\
    =&\frac{1}{2} \sum_{m=0}^{N-1} e^{2 \pi im(1+k) / N}+\frac{1}{2} \sum_{m=0}^{N-1} e^{-2 \pi im(1-k) / N}=\frac{N}{2}\left(\delta_{k, N-1}+\delta_{k, 1}\right)
\end{align*}
which concludes the proof.
\end{proof}
\end{section}

\begin{section}{Identification of $\OA$ with $W^\u\left(\frac{1}{2\pi}\right)$}\label{App_C}
The aim of this appendix is to clarify why, for \eqref{MFL}, we have $\OA=W^\u\left(\frac{1}{2\pi}\right)$. Consider \eqref{MFL} on the Banach space $L^1(\T^1)$ and linearize it around the equilibrium $\frac{1}{2\pi}$. This yields the linear operator
$$L_{\frac{1}{2\pi}}f(\theta)\coloneqq\frac{K}{2\pi}\int_0^{2\pi}\cos(\theta-\varphi)f(\varphi)\d \varphi,$$
which can be rewritten as
$$L_{\frac{1}{2\pi}}f(\theta)
=\cos\theta\frac{K}{2\pi}\int_0^{2\pi}f(\varphi)\cos\varphi\d \varphi
+\sin\theta\frac{K}{2\pi}\int_0^{2\pi}f(\varphi)\sin\varphi\d \varphi.$$
This expression shows that the range of $L_{\frac{1}{2\pi}}$ is contained in $\operatorname{span}\{\cos,\sin\}$. So that, in particular, $L_{\frac{1}{2\pi}}$ has finite rank and is therefore a compact operator. Consequently, its spectrum is discrete and the eigenvectors corresponding to nonzero eigenvalues must lie in $\operatorname{span}\{\cos,\sin\}$. A direct computation yields
$$L_{\frac{1}{2\pi}}\cos\theta=\frac{K}{2}\cos\theta,\qquad 
L_{\frac{1}{2\pi}}\sin\theta=\frac{K}{2}\sin\theta.$$
Hence, by \cite[Theorem 6.8]{brezis2011functional},
$$\sigma\left(L_{\frac{1}{2\pi}}\right)=\{0,K/2\},$$
and in particular $L_{\frac{1}{2\pi}}$ is a sectorial operator. Therefore, by \cite[Lemma 71.2]{sell2002dynamics}, a local unstable manifold $W^u_{\mathrm{loc}}\left(\frac{1}{2\pi}\right)$ exists in $L^1(\T^1)$. Moreover, from \eqref{red_dyn_OA}, we see that for every element $f_{\alpha,\beta}$ of $\OA$, one has
$$\lim_{t\rightarrow-\infty}f_{\alpha,\beta}(\theta)
=f_{\alpha,0}(\theta)=\frac{1}{2\pi},$$
which implies that $\OA\subseteq W^\u\left(\frac{1}{2\pi}\right)$. Since $\OA$ is a two-dimensional manifold, it follows that locally it must coincide with $W^\u_{\mathrm{loc}}\left(\frac{1}{2\pi}\right)$. Finally, since $\OA$ is invariant under the dynamics, we conclude that
$$\OA=W^\u\left(\frac{1}{2\pi}\right),$$
as in the final step of Theorem~\ref{main_theorem}.
\end{section}